\documentclass{amsart}
\usepackage{amsmath,amsthm,amssymb}
\usepackage{lipsum} 
\usepackage[utf8]{inputenc}
\usepackage{esint}
\usepackage[all]{xy}
\usepackage{amsfonts} 
\usepackage{xcolor}
\usepackage{tikz}
\usepackage[colorlinks=true]{hyperref}
\usepackage{bbold}
\usepackage{enumerate}
\usepackage{calc}
\usepackage{bbold}
\usepackage[colorinlistoftodos,prependcaption,textsize=tiny]{todonotes}
\usepackage{todonotes}
\usepackage[shortcuts]{extdash}
\usepackage[style=alphabetic,maxnames=200,maxalphanames=6,giveninits]{biblatex}
\makeatletter
\numberwithin{equation}{section}
\theoremstyle{plain}
        \newtheorem{theorem}{Theorem}[section]

        \newtheorem{assumption}[theorem]{Assumption}
        \newtheorem{proposition}[theorem]{Proposition}
        \newtheorem{lemma}[theorem]{Lemma}
        \newtheorem{corollary}[theorem]{Corollary} 
        \newtheorem{definition}[theorem]{Definition} 
        \newtheorem{remark}[theorem]{Remark}

\newtheorem*{theorem*}{Theorem}
\newtheorem*{definition*}{Definition}
\newtheorem*{proposition*}{Proposition}

\newcommand{\loc}{\text{loc}}
\newcommand{\sgn}{\operatorname{sgn}}

\newcommand{\R}{\mathbb{R}}
\newcommand{\C}{\mathbb{C}}

\newcommand{\T}{\mathbb{T}}
\newcommand{\N}{\mathbb{N}}
\newcommand{\hf}{\mathbb{f}}

\newcommand{\hg}{\mathbb{g}}
\newcommand{\hQ}{\mathbb{Q}}
\newcommand{\cF}{\mathcal{F}}
\newcommand{\cD}{\mathcal{D}}
\newcommand{\cH}{\mathcal{H}}
\newcommand{\Do}{\R^{2d}}
\newcommand{\sd}{S^{d-1}}
\newcommand{\G}{\R^{4d}\times S^{d-1}}
\newcommand\supp{\operatorname{Supp}}
\renewcommand{\d}{\partial} 
\newcommand{\defeq }{\mathop{=}\limits^{\textrm{def}}}
\newcommand{\dd}{{\,\rm d}}

\title[Localised limite for fuzzy Boltzmann equations]{Passing to the limit in fuzzy Boltzmann equations}

\author[M.~Erbar]{Matthias Erbar}
\author[Z.~He]{Zihui He}
\address[M.~Erbar and Z.~He]
{Fakult\"at f\"ur Mathematik, Universit\"at Bielefeld, Postfach 100131, 33501 Bielefeld, Germany}
\email{erbar@math.uni-bielefeld.de, zihui.he@uni-bielefeld.de*}
\keywords{Boltzmann equation, delocalised collision, hard potential}
\subjclass[2020]{35Q20, 82C40}

\date{\today}

\begin{document}

\begin{abstract}
We study a fuzzy Boltzmann equation, where collisions are delocalised and modulated by a spatial kernel. We show that as the spatial kernel converges to a delta distribution, the solutions to these equations converge to renormalised solutions of the inhomogeneous Boltzmann equations.
\end{abstract}

\maketitle


\section{Introduction}
The classical inhomogeneous Boltzmann equation can be written as
\begin{equation}
\label{IBE}
\begin{aligned}
\d_t f+v\cdot \nabla_x f=Q(f,f),\quad (t,x,v)\in [0,T]\times \Do,
\end{aligned}
\end{equation}
where the collision operator $Q(f,f)$ is given by
\begin{equation*}
\label{cla:Q}
Q(f,f)=\int_{\R^d\times \sd}
\big(f(x,v')f(x,v'_*)-f(x,v)f(x,v_*)\big)B(v-v_*,\omega)\dd v_*\dd \omega,
\end{equation*}
with pre- and post-collisional velocities $(v,v_*)$ and $(v',v_*')$ related via
\begin{equation*}
  v' =v - \langle v-v_*,\omega\rangle w\quad \text{and}\quad v_*' = v_* +
  \langle v-v_*\omega\rangle \omega\quad\text{for any } \omega\in S^{d-1},
\end{equation*}
and suitable collision kernel $B:\R^d\times S^{d-1}\to\R_+$.
In this article, we will consider kernels that are locally integrable and are bounded as $0\leq B(v,\omega)\lesssim 1+ |v|^\mu$ for some $\mu\in [0,1]$, usually called soft potentials with angular cut-off in the literature.

Solutions to the Boltzmann equation describe the evolution of the density of a dilute gas of particles interacting by elastic collision. They formally conserve mass, momentum and energy, i.e. $\int(1,v,|v|^2) f_t\dd x\dd v = \int (1,v,|v|^2)f_0\dd x\dd v$ for all $t\geq 0$. Thus the equation is naturally posed in $L^1(\R^{2d})$. However, the mathematical analysis of the Boltzmann equation is highly challenging. One reason for this is the quadratic dependence on $f$ in the collision term $Q(f,f)$ which in particular already prevents a straight forward weak formulation of the equation when assuming only $f\in L^1(\R^{2d})$.
\textcite{DL89b} introduced the concept of renormalised solutions to relax the integrability of the collision term and showed global in time existence for this notion. Roughly speaking, this consists in searching for $f$ solving 
\begin{equation}
\label{HBE}
(\d_t+v\cdot\nabla_x)\log(1+f)=\frac{Q(f,f)}{1+f}   
\end{equation}
in weak form. The renormalised collision operator  ${Q(f,f)}/{1+f}$ can be shown to be sublinear in $f$, and therefore is (locally) integrable. Renormalised solutions were also studied for various kinetic equations, see  \cite{DL88,Lio94,Vil96}. We mention that the hydrodynamic limit of $L^1$-renormalised solutions was established in \cite{GSR04}. The hard potential collision kernel with angular cut-off collision kernels were considered in \cite{DL89b}. Furthermore, \eqref{IBE} was studied in the case of non-cut-off soft potentials in \cite{ADVW00,AV02}, the discussion in smoother settings can be found in \cite{AMUXY11,AMUXY11b,AMUXY12}.
\medskip

The Boltzmann equation \eqref{IBE} describes an emblematic example of a system out of equilibrium. A general framework for such systems has been put forward under the term GENERIC (General Equations for Non-Equilibrium Reversible Irreversible Coupling, see \cite{Oett05}). It describes dynamics that combine both conservative and dissipative effects and takes the abstract form
\begin{equation}
\label{generic}
    \d_t f=\mathsf L \dd\mathsf E+\mathsf M \dd\mathsf S,
\end{equation}
where $\mathsf S$ is a dissipated entropy and $\mathsf M$ is a symmetric positive semi-definite operator, while $\mathsf E$ is a conserved energy and $\mathsf L$ is a Poisson operator. For the Boltzmann equation the Hamiltonian part is the transport term $\mathsf L\dd \mathsf E =-v\nabla_x f$ with energy $\mathsf E = \frac12\int |v|^2 f$ and the dissipative part is the collision term $\mathsf M \dd\mathsf S=Q(f,f)$ with the Boltzmann entropy $\mathsf S = -\int f\log f$ and suitable operators $\mathsf M, \mathsf L$. When the Hamiltonian part vanishes, the resulting equation $\d_t \mathsf f=\mathsf M \dd\mathsf S$ is known as a gradient flow. For the spatially homogeneous Boltzmann equation $\partial_t f=Q(f,f)$ this formal gradient flow interpretation has been made rigorous in \cite{Erb23} in the framework of gradient flows in spaces of probability measures, for which a rich theory is available by now, see e.g. \cite{AGS08}. This amounts to a variational characterisation of solutions as the minimisers $\mathcal L(f)=0$ of a functional $\mathcal L\geq 0$ that is tightly linked to the dynamic large deviations of underlying particle dynamics. GENERIC systems in principle allow for a similar variational characterisation and a rigorous implementation for the inhomogeneous Boltzmann equation could be expected to provide a new handle on its analysis.
In order to make progress in this direction, we proposed in \cite{EH25} an approximate equation, termed fuzzy Boltzmann equation, to reduce technical difficulties. Here particle collisions are delocalised, but many features of the classical Boltzmann equation \eqref{IBE} are preserve, and we were able to establish a variational characterisation built on the GENERIC structure.
The purpose of the current paper is to show that the fuzzy Boltzmann equation indeed converges to the inhomogeneous Boltzmann equation in the limit of localised collisions.

\subsection{Fuzzy Boltzmann equation}
The fuzzy Boltzmann equation that was proposed in \cite{EH25} reads
\begin{equation}
\label{pre:FBE}
\d_t f+v\cdot \nabla_x f=Q^\sigma_{\sf fuz}(f,f),\quad (t,x,v)\in [0,T]\times \Do,
\end{equation}
where the fuzzy collision operator is given by 
\begin{align*}
Q^\sigma_{\sf fuz}(f,f)=\int_{\Do\times S^{d-1}}
&\big(f(x,v')f(x_*,v'_*)-f(x,v)f(x_*,v_*)\big)\\
&B(v-v_*,\omega)\kappa^\sigma(x-x_*)\dd x_*\dd v_*\dd \omega.
\end{align*}
Here, one can think of particles interacting via delocalised collisions, i.e.\ one particle at positions $x$ interacts with a particle at position $x_*$.
The rate at which such collisions happen is modulated by a kernel $\kappa^\sigma(x)=\sigma^{-\frac{d}{2}}\kappa({x}/{\sqrt{\sigma}})$, $\sigma\in(0,1)$, where 
\begin{equation*}
\kappa(x)=\|\exp(-\langle x\rangle)\|_{L^1(\R^d)}^{-1}\exp(-\langle x \rangle). 
\end{equation*}
When $\sigma\to0$, the spatial kernel $\kappa^\sigma$ converges to a Dirac measure, and at least formally, the fuzzy Boltzmann equation \eqref{pre:FBE} converges to a classical Boltzmann equation \eqref{IBE}. The aim of this article is to rigorously establish this convergence.

For a fixed $\sigma>0$, compared to the classical Boltzmann equation \eqref{IBE}, the delocalised collisions in the fuzzy Boltzmann equation indeed relax the difficulty stemming from the quadratic dependence of the collision operator on $f$. Indeed, it can be bounded easily by suitable moments of $f$. Assuming e.g.~$B$ to be bounded for simplicity, one has the following a prior estimate:
\begin{equation}
\label{pri:fuz}
    \|Q_{\sf fuz}(f,f)\|_{L^1(\Do)}\lesssim \|B\|_{L^\infty(\R^d)}\|\kappa^\sigma\|_{L^\infty(\R^d)}\|f\|_{L^1(\Do)}^2.
\end{equation}

As a consequence, the fuzzy Boltzmann equation \eqref{pre:FBE} is structurally similar to the 
homogeneous Boltzmann equation
\begin{equation}
\label{eq:homo}
    \d_t f=Q_{\sf homo}(f,f),\quad (t,v)\in[0,T]\times \R^d,
\end{equation}
where $f$ depends only on time and velocity, and the transport term $v\cdot\nabla_x f$ vanishes. The well-posedness problem for the homogeneous Boltzmann equation \eqref{eq:homo} has been studied in \cite{Ark72,Ark72b,MW99,Lu99,Wen99,LW02}, among others. Based on this, we studied the solvability of the fuzzy Boltzmann equation \eqref{pre:FBE} with the collision kernel satisfying the growth condition $0\leq B(v,\omega)\lesssim 1+ |v|^\mu$.




\subsection{Main results}\label{sec:main-thm}
In \cite{EH25}, we showed the existence of weak solutions to the fuzzy Boltzmann equation \eqref{pre:FBE} with a fixed $\sigma\in(0,1)$ that satisfy the following mass, momentum and energy conservation laws
\begin{equation*}
    \int_{\Do}(1,v,|v|^2)f_t\dd x\dd v= \int_{\Do}(1,v,|v|^2)f_0\dd x\dd v\quad\text{for any } t\in[0,T].
\end{equation*}

For fixed $p,\,q\ge 1$ we define the 
the functional space $L^1_{p,q}(\Do)$ consisting of all functions $f$ for which 
\begin{align*}
\|f\|_{L^1_{p,q}}\defeq \|(\langle x\rangle ^p+\langle v\rangle ^q)f\|_{L^1(\Do)}<+\infty.
\end{align*}

The Boltzmann entropy of $f$ is defined by 
\begin{equation*}
    \cH(f)=\int_{\Do}f\log f\dd x\dd v
\end{equation*}
if $\max(f\log f,0)$ is integrable, otherwise, we set $\cH(f)=+\infty$. Notice that if $f\in L^1_{2,2}(\Do)$, then we have the integrability of $\min(f\log f,0)$, see Proposition \ref{prop:moment:ent} below for details.
The entropy is non-increasing in time and the following entropy identity holds:
\begin{equation*}
\cH(f_t)-\cH(f_0)=\int_0^t  \cD(f_s)\dd s \quad\forall t\in[0,T],  
\end{equation*}
where the entropy dissipation is defined as
\begin{equation*}
  \cD(f)=\frac14\int_{\G} \kappa^\sigma B(f'f_*'-ff_*)\log\frac{f'f_*'}{ff_*}\dd\sigma
\end{equation*}
and $\dd\sigma$ denotes the Hausdroff measure on $\G$, where we use the abbreviations $f_*=f(x_*,v_*)$ and $f_*'=f(x_*,v'_*)$.

Our following main result shows that the fuzzy Boltzmann equation \eqref{pre:FBE} converges to a classical inhomogeneous Boltzmann equation \eqref{IBE} as the spatial kernel $\kappa^\sigma$ converges to a Dirac measure. 
\begin{theorem}
\label{thm:main}
Let $f_0\in L^1_{2,2}(\Do)$ such that  $\cH(f_0)<+\infty$. Let the collision kernel $B$ satisfy the Assumption \ref{CK}.  
Let $f^\sigma\in C([0,T];L^1(\Do)$ be a weak solution of the fuzzy Boltzmann equation \eqref{pre:FBE} such that
    \begin{align*}
        \sup_{t\in[0,T]} \int_{\Do}(1+|x|^2+|v|^2+|\log f^\sigma|)f^\sigma\dd x\dd v&\le C,\\
        \int_0^T\int_{\G}\kappa^\sigma B \big((f^\sigma)'(f^\sigma)'_*-f^\sigma f^\sigma_*\big)\log\frac{(f^\sigma)'(f^\sigma)'_*}{f^\sigma f^\sigma_*}\dd\sigma\dd t&\le C
    \end{align*}
  for some $C=C(T)>0$.

Then there exists a renormalised solution $f\in C([0,T];L^1(\Do))\cap L^\infty([0,T]; L^1_{2,2}(\Do))$ of the 
inhomogeneous Boltzmann equation \eqref{IBE} with initial value $f_0$, for which (up to a subsequence) 
\begin{equation*}
    f^{\sigma}\to f\quad\text{in} \quad C([0,T];L^1(\Do))
\end{equation*}
as $\sigma\to0$.

Moreover, the following entropy inequality holds:
\begin{equation*}
    \cH(f_t)-\cH(f_0)+\int_0^t\cD(f_s)\dd s\le0,\quad 0\le s\le t\le T.
\end{equation*}
\end{theorem}

The definitions of weak and renormalised solutions are given in Definition~\ref{def:sols}. The existence of the weak solutions to the fuzzy Boltzmann equations satisfying the assumptions in Theorem \ref{thm:main} are given in Theorem~\ref{thm:existence}.

Denote by $f^\sigma$ the solution of the fuzzy Boltzmann equation \eqref{pre:FBE} corresponding to  $\kappa^\sigma$. Define the function
\begin{equation}
\label{def:hf}
\hf^\sigma(x,v)\defeq  f^\sigma*_x\kappa^\sigma=\int_{\R^d}f^\sigma(x_*,v)\kappa^\sigma(x-x_*)\dd x_*,
\end{equation}
and notice that $\|\hf^\sigma\|_{L^\infty_xL^1_v}\le \|\kappa^\sigma\|_{L^\infty(\R^d)}\|f^\sigma\|_{L^1(\Do)}$.

The fuzzy Boltzmann equation \eqref{pre:FBE} takes the form
\begin{equation}
\label{FBE}
(\d_t+v\cdot\nabla_x)f^\sigma=Q(f^\sigma,\hf^\sigma)= Q^\sigma_{\sf fuz}(f^\sigma,f^\sigma).
\end{equation}

\textcite{DL89} and  \textcite{Lio94} showed the compactness and stability of the inhomogeneous Boltzmann equation \eqref{IBE} in the sense that if ${f^n}$ is a sequence of renormalised solutions with initial values $f^n_0$, then $f^n_0 \to f_0$ in $L^1(\Do)$ if and only if $f^n$ converges strongly to a renormalised solution $f$ with initial value $f_0$.

We follow their compactness argument to prove Theorem~\ref{thm:main} about the convergence of the fuzzy Boltzmann equation. The remaining sections are organised as follows:
In Section~\ref{sec:pre}, we discuss different types of solutions and their properties. Moreover, we show that (up to a subsequence)
\begin{equation*}
    f^\sigma\rightharpoonup f\quad\text{and}\quad \hf^\sigma\rightharpoonup\hf \quad\text{in}\quad L^1([0,T]\times\Do)
\end{equation*}
as $\sigma\to0$ for some $f$ and $\hf\in L^1([0,T]\times\Do)$.
In Section~\ref{sec:weak}, we apply the velocity averaging technique to \eqref{FBE} with $\sigma$ to derive strong compactness results.
In Section~\ref{sec:strong}, we show the strong convergence $f^\sigma\to f$ in $C([0,T];L^1(\Do))$, and as a consequence, the equality $f=\hf$. We show that $f$ is indeed a renormalised solution of the inhomogeneous Boltzmann equation.

\subsection*{Acknowledgements}
Funded by the Deutsche Forschungsgemeinschaft (DFG, German Research Foundation) – Project-ID 317210226 – SFB 1283.

\section{Solvability and compactness results}
\label{sec:pre}

In Subsection~\ref{sec:def-prop}, we will introduce definitions,  properties, and solvability of fuzzy Boltzmann equations \eqref{FBE}. In Subsection~\ref{sec:weak-cmpt}, we will show the weak compactness of approximation quantities.

\subsection{Definition and properties of solutions}
\label{sec:def-prop}
Throughout this article, we use the following abbreviation
\begin{align*}
f=f(x,v),\quad f_*=f(x,v_*),\quad f'=f(x,v'),\quad f'_*=f(x,v'_*)\;.
\end{align*}
Occasionally, we use the abbreviation $f_*=f(x_*,v_*)$, $f'_*=f(x_*,v'_*)$ instead. This will be explicitly highlighted and no confusion should arise.

Throughout this paper, we make the following assumptions on the 
collision kernel $B$.
\begin{assumption}\label{CK}
Let $\mu\in [0,1]$.
The collision kernel $B(v,w)\in L^1_\loc (\R^d\times S^{d-1};\R_+)$,  depends only on $|v|$ and $|\langle v,w\rangle|$, 
and satisfies
the following growth condition
\begin{equation*}
    0 \le B(v,w)\le C\langle v\rangle ^\mu\quad \forall v\in \R^d,\quad w  \in S^{d-1}.
\end{equation*}
\end{assumption}

We will introduce weak and renormalised solutions of fuzzy Boltzmann equations, where the definition of renormalised solutions for inhomogeneous Boltzmann equations \eqref{IBE} were introduced in \cite{DL89b}.
\begin{definition}
\label{def:sols}
For any given initial value $f_0\in L^1(\Do)$, we say $f\in C([0,T]; L^1(\Do))$ 
is a \emph{weak solution} of \eqref{FBE}, if
    \begin{equation*}
(\d_t+v\cdot\nabla_x)f^\sigma=Q(f^\sigma,\hf^\sigma)
    \end{equation*}
    holds in the distribution sense, where $\hf^\sigma$ is defined as in \eqref{def:hf}.
    
We say $f$ is a \emph{renormalised solution} of \eqref{FBE} for $\alpha\in(0,1)$, if
     \begin{equation}
     \label{FBE:renorm}
(\d_t+v\cdot\nabla_x)g^{\sigma,\alpha}=Q^{\alpha}(f^\sigma,\hf^\sigma)
    \end{equation}
holds in the  distribution sense, where 
\begin{equation}
\label{def:g}
g^{\sigma,\alpha}\defeq\alpha^{-1}\log(1+\alpha f^\sigma)\quad\text{and}\quad Q^{\alpha}(f^\sigma,\hf^\sigma)\defeq \frac{Q(f^\sigma,\hf^\sigma)}{1+\alpha f^\sigma}.
\end{equation}

\end{definition}

Let $B^z_R\defeq\{z\in \R^k\mid |z|\le R\}$ be the ball centred at $0$ with radius $R$. 

\begin{remark}
\begin{itemize}
    \item The weak and renormalised solutions of the fuzzy Boltzmann equation are equivalent when the collision kernel $B$ satisfies Assumption~\ref{CK}. The proof can follow \cite{DL89}, and the key point is that the gain and loss parts of the collision term are local integrable $Q^\pm(f^\sigma,\hf^\sigma)\in L^1_\loc(\Do)$, i.e. 
\begin{equation*}
\label{apriori}
\|Q^\pm(f^\sigma,\hf^\sigma)\|_{L^1(\R^d\times B^v_R)}\lesssim R \|\kappa^\sigma\|_{L^\infty(\R^d)}\|f^\sigma\|_{L^1(\Do)}^2.
\end{equation*}

    \item Let $g^{\sigma,\alpha}$ be defined as in \eqref{def:g}. We define
\begin{equation*}
    \hg^{\sigma,\alpha}\defeq g^{\sigma,\alpha}*_x \kappa^\sigma\quad \text{and}\quad
\hQ^{\sigma,\alpha}(f^\sigma,\hf^\sigma)\defeq  Q^{\alpha}(f^\sigma,\hf^\sigma)*_x\kappa^\sigma.
\end{equation*}
If $f$ is a renormalised solution to the fuzzy Boltzmann equation \eqref{FBE}, then for any fixed $\alpha\in (0,1)$
\begin{equation}
\label{FBE:tilde}
(\d_t+v\cdot\nabla_x) \hg^{\sigma,\alpha}= \hQ^{\sigma,\alpha}_*(f^\sigma,\hf^\sigma)
    \end{equation}
    holds in the distribution sense.  Indeed, for any test function $\varphi\in C^\infty_c([0,T)\times\Do)$, we have
\begin{equation*}
\int_{\Do} \hg^{\sigma,\alpha}\varphi\dd x\dd v=\int_{\Do} g^{\sigma,\alpha}\big(\varphi*_x \kappa^\sigma\big)\dd x\dd v.
\end{equation*}
Thus, \eqref{FBE:renorm} holding in distribution sense ensures that  \eqref{FBE:tilde} holds in distribution sense.
\end{itemize}
\end{remark}

We have the following solvability results.
\begin{theorem}
\label{thm:existence}
Let the collision kernel $B$ satisfy Assumption~\ref{CK}. 
If $f_0\in L^1_{2,2}(\Do;\R_+)$ and  $\cH(f_0)<+\infty$, 
    then for any $T>0$ there exists a global-in-time  weak solution $f^\sigma\in C([0,T];L^1(\Do;\R_+))\cap L^\infty([0,T];L^1_{2,2}(\Do;\R_+))$ such that the following  mass, momentum and energy conservation laws hold
    \begin{align*}
    \int_{\Do}(1,v,|v|^2)f^\sigma_t\dd x\dd v=    \int_{\Do}(1,v,|v|^2)f_0\dd x\dd v.
    \end{align*}
    Moreover, we have the following entropy inequality
\begin{equation}
\label{ineq:EE}
\begin{aligned}
\cH(f^\sigma)-\cH(f_0)+\int_0^T\cD(f^\sigma)\dd t\le0
\end{aligned}
\end{equation}
for all $t\in[0,T]$,
and we have the estimate $\|f^\sigma_t\|_{L^1_{2,0}(\Do)}\lesssim_T \|f_0\|_{L^1_{2,2}(\Do)}$.

\end{theorem}


\begin{remark}
\begin{itemize}
    \item 

    When $\mu=0$, the weak solutions in Theorem \ref{thm:existence} are unique, and the entropy identity holds, i.e. \eqref{ineq:EE} hold with equal signs, which has been shown in \cite{EH25}. 
    
     \item (Energy conservation law.) When $\mu\in(0,1]$, the solvability results have been shown in \cite{EH25} with stronger initial values in $L^1_{2,2+\mu}(\Do)$, where we obtain energy and entropy identities in \eqref{ineq:EE}. One can straightforwardly follow the arguments in \cite{EH25} to derive the solvability with $f_0\in L^1_{2,2}(\Do)$, for which we only have energy and entropy inequalities, i.e.
     \begin{equation*}
\int_{\Do}|v|^2f_t\dd x\dd v\le\int_{\Do}|v|^2f_0\dd x\dd v\quad\forall t\in[0,T].
\end{equation*} 
One can show the energy conservation laws
by following the arguments for homogeneous Boltzmann equations \eqref{eq:homo} in \cite{Lu99} to show that the energy is non-decreasing over time. Combining the energy inequality \eqref{ineq:EE}, we conclude the energy conservation laws. We put the detailed proofs in the Appendix.

\item (Uniqueness.) When $\mu\in(0,1]$, in the case of the collision kernels take the precise form
\begin{equation*}
    B(v,w)=|v|^\mu b(\theta),
\end{equation*}
where $\theta$ denotes the collision angle $\theta=\arccos\frac{\langle v,w\rangle}{|v|}$ and $b$ is a bounded function. By following the results in \cite{MW99} for homogeneous Boltzmann equations \eqref{eq:homo}, one can show the uniqueness of the energy-conserved solutions on the domain $\T^3\times\R^3$. We put the detailed proof in the Appendix.
\end{itemize}
\end{remark}

We summarise properties of the weak solutions for fuzzy Boltzmann equation \eqref{FBE}. We recall the definition 
$\hf^\sigma=f^\sigma*_x \kappa^\sigma$. 
\begin{proposition}\label{prop:moment:ent}
Let $\sigma\in(0,1)$. Let the collision kernel $B$ satisfy Assumption~\ref{CK}. Let $f^\sigma\in C([0,T];L^1(\Do;\R_+))$ be a weak solution of the fuzzy Boltzmann equation \eqref{FBE} with initial value  $f_0\in L^1_{2,2}(\Do;\R_+)$ such that $|\cH(f_0)|<+\infty$. For all $T>0$, we have
\begin{equation}
\label{bdd:L122}
    \begin{aligned}
&\|f^\sigma_t\|_{L^1(\Do)}=\|\hf^\sigma_t\|_{L^1(\Do)}=\|f_0\|_{L^1(\Do)},\\
&\||v|^2f^\sigma_t\|_{L^1(\Do)}=\||v|^2\hf^\sigma_t\|_{L^1(\Do)}\le\||v|^2f_0\|_{L^1(\Do)},\\
&\|f^\sigma_t\|_{L^1_{2,0}(\Do)}\lesssim_T\|f_0\|_{L^1_{2,2}(\Do)}.
        \end{aligned}
            \end{equation}
The following entropy inequality holds
\begin{equation}
\label{ineq:cH}
    \cH(f^\sigma_t)-\cH(f_0)\le-\int_0^t \cD(f^\sigma_s)\dd s\quad\forall t\in[0,T],
\end{equation}
where $f_*=f(x_*,v_*)$ and $f'_*=f(x_*,v_*')$, and
\begin{equation}
\label{def:dissipation}
\cD(f)\defeq \frac14\int_{\R^{4d}\times S^{d-1}}\kappa^\sigma B\big(f'f'_*-ff_*\big)\big(\log(f'f'_*)-\log(ff_*)\big)\dd\sigma.
\end{equation}
In addition, if the collision kernel $\langle v\rangle^{\mu}\lesssim B$ has a polynomial lower bound, then equality holds in \eqref{ineq:cH}.

The following uniform bounds hold   \begin{equation}
\label{bdd:H-D}
\sup_{\sigma}\Big(\|f^\sigma_t\log f^\sigma_t\|_{L^\infty([0,T];L^1(\Do))}+\int_0^T\cD(f^\sigma_t)\dd t\Big)
\le C
        \end{equation}
        for some constant $C=C(T)>0$ depending only on $T$.
\end{proposition}
\begin{proof}
The conservation laws and bounds \eqref{bdd:L122} with $f^\sigma$ and entropy identity \eqref{ineq:cH} have been showed in \cite{EH25}. And \eqref{bdd:L122} hold with $\hf^\sigma$ by definition.

To show \eqref{bdd:H-D}, we use the following estimate, see for example \cite{JKO98}
\begin{align*}
\|f^\sigma_t\log f^\sigma_t\|_{L^1(\Do)}\le \cH(f^\sigma_t)+C\|f^\sigma_t\|_{L^1_{2,2}(\Do)}
\end{align*}
for all $\sigma\in(0,1)$ and $t\in[0,T]$.
We substitute the entropy inequality \eqref{ineq:cH} and the bounds on $\|f\|_{L^1_{2,2}(\Do)}$ to derive
\begin{align*}
\|f^\sigma_t\log f^\sigma_t\|_{L^1(\Do)}+\int_0^t\cD(f^\sigma_\tau)\dd \tau\le \cH(f_0)+C(T)\|f_0\|_{L^1_{2,2}(\Do)}    
\end{align*}
for all $t\in[0,T]$.
\end{proof}

\begin{remark}
    \begin{itemize}
    
        \item In \cite{DL89,Lio94}, DiPerna and Lions worked with a more general collision kernel satisfying the assumptions 
        \begin{equation}
        \label{CK:DL}
             B\in L^1_\loc(\R^d\times S^{d-1})\quad\text{and}\quad \lim_{|v|\to\infty}\langle v\rangle^{-2}\int_{\{|z-v|\le R\}}A(z)\dd z=0, 
        \end{equation}
        where $A$ is defined as 
        \begin{equation*}
        A(z)\defeq \int_{S^{d-1}} B(z,w)\dd\omega.    
        \end{equation*}
        We note that the collision kernels satisfying Assumption~\ref{CK} also satisfy the assumption \eqref{CK:DL}.
        \item The convergence results in Theorem~\ref{thm:main} hold also for the l collision kernel satisfying \eqref{CK:DL}. One can adapt our proof by approximating the collision kernel $B(z,w)$ by $\min(B,N) \mathbb{1}_{\{|z|\le N\}}\in L^1\cap L^\infty(\R^d\times S^{d-1})$ and pass to the limit.
        
        \item Similar to the inhomogeneous Boltzmann equations in \cite{DL89}, the compactness arguments Section~\ref{sec:pre}-\ref{sec:strong} can also be used to show the existence of solutions of the fuzzy Boltzmann equations \eqref{FBE} for any fixed $\sigma\in(0,1)$. One can construct a sequence of approximated solutions of \eqref{FBE} with a smooth and truncated collision kernel, the compactness results ensures the existence of the weak limits, which solves \eqref{FBE}.
    \end{itemize}
\end{remark}

\subsection{Weak compactness results}
\label{sec:weak-cmpt}
For fixed $\alpha\in(0,1)$ and all $\sigma\in(0,1)$, we will show the weak relative compactness of 
\begin{equation}
\label{weak-cmpt:f}
\{f^\sigma\},\quad\{\hf^\sigma\},\quad\{g^{\sigma,\alpha}\}\quad\text{and}\quad\{\hg^{\sigma,\alpha}\}\quad\text{in}\quad L^1([0,T]\times \Do),
\end{equation}
and, for all $R\ge0$, the weak relative compactness of 
\begin{equation}
\label{weak-cmpt:Q}
\{Q^\alpha(f^\sigma,\hf^\sigma)\}\quad\text{and}\quad \{\hQ^\alpha(f^\sigma,\hf^\sigma)\}\quad\text{in}\quad L^1([0,T]\times \R^d\times B^v_R).
\end{equation}


We use the following Dunford--Pettis theorem to show the weak compactness result.
     \begin{theorem}[Dunford--Pettis]
     \label{thm:DP}
     Let $(\Omega,\mu)$ be a measure space. Then $\{h^\gamma\}_{\gamma\in(0,1)}\subset L^1(\Omega)$ is weakly relatively compact if and only if
     \begin{enumerate}
         \item $\{h^\gamma\}$ is uniformly bounded in $L^1(\Omega)$;
         \item $\{h^\gamma\}$ is equi-integrable, i.e. for any $\varepsilon>0$, there is $\delta>0$ such that for any measurable $A\subset \Omega$ with $\mu(A)<\delta$, we have
         \begin{equation*} \sup_{\gamma}\int_{A}|h^\gamma|\dd\mu<\varepsilon.
         \end{equation*}
         \item for any $\varepsilon>0$, there is $D\subset\Omega$ with $\mu(D)<+\infty$ such that
         \begin{equation*}
          \sup_{\gamma}\int_{D^c}|h^\gamma|\dd\mu<\varepsilon.
         \end{equation*}
     \end{enumerate}
     \end{theorem}

We first show the weak compactness of \eqref{weak-cmpt:f}.  
\begin{lemma}\label{lemma:f}
Let $\alpha\in(0,1)$ be fixed. The sequences 
\begin{equation*}
    \{f^\sigma\},\quad\{\hf^\sigma\},\quad\{g^{\sigma,\alpha}\}\quad\text{and}\quad\{\hg^{\sigma,\alpha}\}
\end{equation*}
are weakly relatively compact in $L^1([0,T]\times\Do)$.
\end{lemma}
\begin{proof}
We first show $f^\sigma$ is weakly relatively compact in $L^\infty([0,T];L^1(\Do))$. We check the conditions in Theorem~\ref{thm:DP}.
\begin{enumerate}
    \item $f^\sigma$ is uniformly bounded in $L^1(\Do)$, since  $\|f^\sigma_t\|_{L^1(\Do)}=\|f_0\|_{L^1(\Do)}$ for all $t$ and $\sigma$.
    \item Let $\Phi(z)=z(\log z)^+$.
We have 
\begin{equation}
\label{equi-int:f-sigma}
    \sup_{t,\sigma}\int_{\Do} \Phi(f^\sigma)\le  \sup_{t,\sigma}\| f^\sigma\log f^\sigma\|_{L^1(\Do)}\le C(T)
\end{equation}
as in Proposition~\ref{prop:moment:ent}, which ensures the equi-integrability of $f^\sigma$.
    \item We observe that
    \begin{equation*}
        \int_{(B^{x,v}_{\sqrt 2 R})^c}f^\sigma\dd x\dd v\le \int_{\R^d\times(B^v_{R})^c}f^\sigma\dd x\dd v+\int_{(B^x_{R})^c\times \R^d}f^\sigma\dd x\dd v.
    \end{equation*}
    By using of the uniform bounds on $\|f^\sigma\|_{L^\infty([0,T];L^1_{2,2}(\Do))}$ as in Proposition~\ref{prop:moment:ent}, we have 
    \begin{equation}
    \label{int:f:vani}
    \begin{aligned}    &\sup_{t,\sigma}\Big(\int_{\R^d\times (B^v_R)^c} f^\sigma\dd x\dd v+\int_{(B^x_R)^c\times \R^d} f^\sigma\dd x\dd v\Big)\\
    \le&{} 2R^{-2}\|f_0\|_{L^1_{2,2}(\Do)}\to0    \quad\text{as }R\to+\infty.
    \end{aligned}  
    \end{equation}
    \end{enumerate}
   We conclude that $\{f^\sigma\}$ is weakly relative compact in $L^\infty([0,T];L^1(\Do))$.

To show the weak compactness of $\hf^\sigma$, we recall the definition
\begin{align*}
\hf^\sigma=f^\sigma*_x \kappa^\sigma=\int_{\R^d}f^\sigma(x_*)\kappa^\sigma(x-x_*)\dd x_*.
\end{align*}
We check the conditions in Theorem~\ref{thm:DP}. $\hf^\sigma$ is obviously uniformly bounded in $L^1$. By using of Jensen's inequality, we have 
    \begin{equation}
    \label{ineq:Phi:hf}
        \|\Phi(\hf^\sigma)\|_{L^1(\Do)}\le \int_{\Do}\Phi(f^\sigma)*_x\kappa^\sigma\dd x\dd v\le \|f^\sigma\log f^\sigma\|_{L^1(\Do)},
    \end{equation}
    and we conclude $(2)$ as before.
   The condition $(3)$ is a consequence of 
\begin{equation}
\label{id:h}
\int_{\R^d\times (B^v_R)^c} \hf^\sigma\dd x\dd v=\int_{\R^d\times (B^v_R)^c}f^\sigma\dd x\dd v, 
\end{equation}
and 
\begin{equation}
\label{hf:no-leak:x}
\begin{aligned}
    &\int_{(B^x_R)^c \times\R^d}\hf^\sigma=\int_{\R^{3d}} f^\sigma(x_*,v)\kappa^\sigma(x-x_*)\mathbb{1}_{\{|x|\ge R\}}\dd x\dd x_* \dd v\\
     \le&{}\int_{\R^{3d}} f^\sigma(x_*,v)\kappa^\sigma(y)\big(\mathbb{1}_{\{|x_*|\ge {R}/{2}\}}+\mathbb{1}_{\{|y|\ge {R}/{2}\}}\big)\dd y\dd x_* \dd v\\
     \le&{} \int_{(B^x_{R/2})^c\times\R^d}f^\sigma+\int_{B_{R/2}^c}\kappa^\sigma\to0\quad\text{as}\quad R\to+\infty
\end{aligned}
    \end{equation}
 uniformly in $\sigma$, where we use the convergence \eqref{int:f:vani}.

To show the weak compactness of $g^{\sigma,\alpha}$ and $\hg^{\sigma,\alpha}$, we recall the definition
\begin{align*}
g^{\sigma,\alpha}=\log(1+\alpha f^\sigma)\quad\text{and}\quad \hg^{\sigma,\alpha}=g^{\sigma,\alpha}*_x\kappa^\sigma.
\end{align*}
Notice that 
    \begin{equation*}
    0\le g^{\sigma,\alpha}\le f^\sigma\quad\text{and}\quad 0\le \hg^{\sigma,\alpha}\le \hf^\sigma.   
    \end{equation*}
Thus the weak compactness of $g^{\sigma,\alpha}$ and $\hg^{\sigma,\alpha}$ follow from the weak compactness of  $f^\sigma$ and $\hf^\sigma$.

\end{proof}

 We define the renormalised gain and loss collision terms as follows
\begin{align*}
Q^{\alpha,+}(f^{\sigma},\hf^\sigma)&
\defeq\frac{1}{1+\alpha f^\sigma}\int_{\R^d\times S^{d-1}}B(f^{\sigma})'(\hf^{\sigma})_*'\dd v_*\dd \omega,\\
Q^{\alpha,-}(f^{\sigma},\hf^\sigma)&\defeq\frac{f^{\sigma}}{1+\alpha f^\sigma}\int_{\R^d\times S^{d-1}}B\hf^{\sigma}_*\dd v_*\dd\omega\\
&=\frac{f^{\sigma}}{1+\alpha f^\sigma}L(\hf^\sigma),
\end{align*}
where we define 
\begin{align*}
L(f)(x,v)\defeq \int_{\R^d}f(x,v_*)A(v-v_*)\dd v_*\quad\text{and}\quad A(z)=\int_{S^{d-1}}B(z,\omega)\dd \omega.
\end{align*}
Correspondingly, we define 
\begin{equation*}
\hQ^{\alpha,+}(f^{\sigma},\hf^\sigma)\defeq Q^{\alpha,+}(f^{\sigma},\hf^\sigma)*_x\kappa^\sigma\quad\text{and}\quad  \hQ^{\alpha,-}(f^{\sigma},\hf^\sigma)\defeq Q^{\alpha,-}(f^{\sigma},\hf^\sigma)*_x\kappa^\sigma.    
\end{equation*}


\begin{lemma}\label{lemma:Q}
Let $\alpha\in(0,1)$ be fixed. Let $R\ge 0$. The sequences 
\begin{equation*}
    \{Q^{\alpha,\pm}(f^\sigma,\hf^\sigma)\}\quad\text{and}\quad\{\hQ^{\alpha,\pm}(f^\sigma,\hf^\sigma)\}
\end{equation*}
are weakly relatively compact in $L^1([0,T]\times\R^d\times B^v_R)$.
\end{lemma}

\begin{proof}
We first show the weak compactness of the loss term $Q^{\alpha,-}(f^\sigma,\hf^\sigma)$. It is enough to show the weak compactness of $\{L(\hf^\sigma)\}$, since  $0\le Q^{\alpha,-}(f^\sigma,\hf^\sigma)\le \alpha^{-1}L(\hf^\sigma)$. 

To show the weak compactness of $L(\hf^\sigma)$, we first approximate the collision kernel $A$ by 
\begin{equation}
\label{Ak}
A_k(v)\defeq A(v)\mathbb{1}_{\{|v|\le k\}} \in L^\infty\cap L^1(\R^d)   
\end{equation}
for $k\in \N_+$. We first show the weak compactness of 
\begin{equation*}
    L_k(\hf^\sigma)(x,v)\defeq \int_{\R^d}\hf^\sigma(x,v_*)A_k(v-v_*)\dd v_*.
\end{equation*}
We verify the conditions in Theorem~\ref{thm:DP}.
\begin{enumerate}
       \item The uniform bound of $\|L_k(\hf^\sigma)\|_{L^1(\Do)}$ is given by
\begin{equation*}
    \int_{\Do}L_k(\hf^\sigma)\dd x\dd v=\|A_k\|_{L^1(\R^d)}\|f_0\|_{L^1(\Do)}.
\end{equation*}
         \item Let $\Phi(z)=z(\log z)^+$ and $a_k=\|A_k\|_{L^1(\R^d)}$. Then we have
\begin{equation*}
    \Phi(z)\le a_k\Phi(\frac{z}{a_k})+ z\log a_k.
\end{equation*}
By using the convexity of $\Phi$ and Jensen's inequality, we have
\begin{equation}
    \label{bdd:Phi:Lf}
    \begin{aligned}
    &\int_{\Do} \Phi(L_k(\hf^\sigma))\dd x\dd v\\
    \le&{} a_k\int_{\Do}\Phi({L_k(\hf^\sigma)}/{a_k})\dd x\dd v+|\log a_k|\int_{\Do}L_k(\hf^\sigma)\dd x\dd v\\
\le&{}  \|L_k(\Phi(\hf^\sigma))\|_{L^1(\Do)}+a_k|\log a_k|\|f^\sigma\|_{L^1(\Do)}\\
\le&{}  a_k\|\Phi(\hf^\sigma)\|_{L^1(\Do)}+a_k|\log a_k|\|f^\sigma\|_{L^1(\Do)}\\
\le&{} a_k\|f^\sigma\log f^\sigma\|_{L^1(\Do)}+a_k|\log a_k|\|f_0\|_{L^1(\Do)},
\end{aligned}
\end{equation}
where we use the inequality \eqref{ineq:Phi:hf} for the last inequality. The right-hand side of \eqref{bdd:Phi:Lf} are uniformly bounded as a consequence of  Proposition~\ref{prop:moment:ent}.
         
         \item We only need to show the smallness of 
         \begin{equation*}
             \int_{\R^d\times (B_R^v)^c}L_k(\hf^\sigma)\dd x \dd v\quad\text{and}\quad \int_{(B^x_R)^c\times \R^d}L_k(\hf^\sigma)\dd x \dd v.
         \end{equation*}
        Indeed, we have
         \begin{align*}
        & \int_{\R^d\times (B^v_R)^c} L_k(\hf^\sigma)\dd x\dd v\\
        \le&{} \int_{\R^{3d}} \hf^\sigma (x,v_*)A_k(v-v_*) \mathbb{1}_{\{|v|\ge R\}}\dd x\dd v\dd v_*\\
\le&{}\int_{\R^{3d}} \hf^\sigma (x,v_*)A_k(z) (\mathbb{1}_{\{|v_*|\ge {R}/{2}\}}+\mathbb{1}_{\{|z|\ge {R}/{2}\}})\dd z\dd x\dd v_*\\
         \le&{}\| \hf^\sigma\|_{L^1(\Do)}{\int_{\R^d}A_k(z) \mathbb{1}_{\{|z|\ge {R}/{2}\}}\dd z}\\
         &+\|A_k\|_{L^1(\R^d)}\int_{\Do} \hf^\sigma (x,v) \mathbb{1}_{\{|v|\ge {R}/{2}\}}\dd x\dd v\to0
         \end{align*}
         as $R\to\infty$ uniformly in $\sigma$. Notice that as $R\to+\infty$,   the integrability of $A$ ensures $\|A_k\|_{L^1(B^c_{{R}/{2}})}\to0$, and 
         \begin{align*}
             &\int_{\R^{d}\times (B^v_{{R}/{2}})^c} \hf^\sigma \dd x\dd v\lesssim {R}^{-2}\|f^\sigma\|_{L^1_{0,2}(\Do)}\to0.
         \end{align*}
         As a consequence of \eqref{hf:no-leak:x}, we have 
         \begin{equation*}
             \begin{aligned}
                 \sup_\sigma\int_{(B^x_{R/2})^c\times\R^d}L_k(\hf^\sigma)\le \|A_k\|_{L^1(\R^d)}\sup_\sigma\int_{(B^x_{R/2})^c\times\R^d}\hf^\sigma\to0
             \end{aligned}
         \end{equation*}
         as $R\to\infty$ uniformly in $\sigma$.
\end{enumerate}
We conclude with the weak relative compactness of $L_k(\hf^\sigma)$.

To pass to the limit by letting $k\to\infty$, we first verify that $L(\hf^\sigma)\in L^\infty([0,T];L^1(\R^d\times B^v_R))$. Indeed, we have 
\begin{align*}
\int_{ \R^d\times B^v_R}L(\hf^\sigma)\dd x\dd v &=\int_{\Do}\hf^\sigma(x,v_*)\Big(\int_{B^v_R} A(v-v_*)\dd v\Big)\dd x\dd v_*\\
&\lesssim \int_{\Do}\hf^\sigma\langle v\rangle^\mu\dd x\dd v\le \|f^\sigma\|_{L^1_{0,2}(\Do)}.
\end{align*}

We are left to show that
\begin{equation*}
    L_k(\hf^\sigma)\to L(\hf^\sigma)\quad\text{as} \quad k \to\infty
\end{equation*}
uniformly in $\sigma\in(0,1)$.
Without loss of generality, we assume $k\ge R$, then
\begin{equation}
\label{k:ge:R}
\begin{aligned}
&\|L_k(\hf^\sigma)-L(\hf^\sigma)\|_{L^1(\R^d\times B^v_R)}\\
=&{}\int_{\Do\times B^v_R}\hf^\sigma(x,v_*)A(v-v_*)\mathbb{1}_{\{|v-v_*|\ge k\}}\dd x\dd v_* \dd v\\
\le&{} \int_{\Do}\hf^\sigma(x,v_*)\mathbb{1}_{\{|v_*|\ge k-R\}}(\int_{B^v_R} A(v-v*)\dd v\big)\dd v_*\dd x\\
\lesssim&{}\int_{\R^d\times(B^v_{k-R})^c}f^\sigma(x,v)\langle v\rangle^\mu\dd x\dd v\to 0
\end{aligned}
\end{equation}
as $k\to\infty$ uniformly in $\sigma$, since we have uniformly energy bound \eqref{bdd:L122} and $0\le\mu<1$. 

Hence, we have the weak relative compactness of $L(\hf^\sigma)$ and $Q^{\alpha,-}(f^\sigma,\hf^\sigma)$.     
    
We show the weak compactness of the gain term $Q^{\alpha,+}(f^\sigma,\hf^\sigma)$. We use the following algebraic inequality, see for example \cite{Ark84}
\begin{align}
a\le c b +{\frac{a-b}{\log c}\log\frac{a}{b} }\quad\forall \, 0\le a,\,b\text{ and }c>1.   \label{alg:ineq:abc}
\end{align}
We briefly show \eqref{alg:ineq:abc}, if $a\le cb$, then the inequality holds since $(a-b)\log\frac{a}{b}\ge 0$; If $a>cb$, then $cb+\frac{a-b}{\log c}\log\frac{a}{b}\ge a+(c-1)b\ge a$.

We substitute $(f^\sigma)_*'(f^\sigma)'=f^\sigma(x_*,v_*')f^\sigma(x,v')$ and $f^\sigma_*f^\sigma=f^\sigma(x_*,v_*)f^\sigma(x,v)$ to the above inequality and convoluted by $\kappa^\sigma$ to derive
\begin{equation}
\label{relation:f':f}
\begin{aligned}
    (f^\sigma)'(\hf^\sigma)_*'\le C f^\sigma\hf^\sigma_*+(\log C)^{-1}\Big(((f^\sigma)_*'(f^\sigma)'-f^\sigma_*f^\sigma)\log\frac{(f^\sigma)_*'(f^\sigma)'}{f^\sigma_*f^\sigma}\Big)*_x\kappa^\sigma
        \end{aligned}
\end{equation}
for all constant $C>1$.
 And hence, we have the comparison inequality
 \begin{equation}
 \label{bdd:Q-:Q+}
\begin{aligned}
    &Q^{+}(f^\sigma,\hf^\sigma)\le C Q^{-}(f^\sigma,\hf^\sigma)+\frac{h(f^\sigma)}{\log C}\quad \text{and}\\
    &Q^{\alpha,+}(f^\sigma,\hf^\sigma)\le C Q^{\alpha,-}(f^\sigma,\hf^\sigma)+\frac{h(f^\sigma)}{\log C},
\end{aligned}
\end{equation}
where 
\begin{align*}
h(f^\sigma)\defeq \int_{\Do\times S^{d-1}}\log\frac{(f^\sigma)'_*(f^\sigma)'}{f^\sigma_*f^\sigma} ({(f^\sigma)'_*(f^\sigma)'}-{f^\sigma_*f^\sigma}) \kappa^\sigma B\dd x_*\dd v_*\dd\omega.
\end{align*}
Notice that the uniform bound has been shown in Proposition~\ref{prop:moment:ent}
$$\|h(f^\sigma_t)\|_{L^1([0,T]\times\Do)}=\int_0^TD(f^\sigma_t)\dd t\le C(T).$$

By using the weak compactness of $Q^{\alpha,-}(f^\sigma,\hf^\sigma)$ and choosing  $C$ large enough in  \eqref{bdd:Q-:Q+}, we conclude the weak compactness of $Q^{\alpha,+}(f^\sigma,\hf^\sigma)$.

The weak compactness of $\{\hQ^{\alpha,-}(f^{\sigma},\hf^\sigma)\}$ follows from the weak compactness of $\{Q^{\alpha,-}(f^{\sigma},\hf^\sigma)\}$ by using of \eqref{ineq:Phi:hf} and\eqref{id:h}. Similar to \eqref{bdd:Q-:Q+}, we have the following comparison inequality
    \begin{equation*}
        {\hQ^{\alpha,+}(f^\sigma,\hf^\sigma)}\le C {\hQ^{\alpha,-}(f^\sigma,\hf^\sigma)}+\frac{h(f^\sigma)*_x\kappa^\sigma}{\log C}.
    \end{equation*}
   Since $\|h(f^\sigma)*_x\kappa^\sigma\|_{L^1([0,T]\times\Do)}=\|h(f^\sigma)\|_{L^1([0,T]\times\Do)}$ are uniformly bounded, we have the weak compactness of $\{{\hQ^{\alpha,+}(f^\sigma,\hf^\sigma)}\}$ by choosing the constant $C$ large enough.

\end{proof}
\begin{corollary}
\label{coro:weak-cpt:quo}
The sequence $\{\frac{Q^{\pm}(f^\sigma,\hf^\sigma)}{1+L(\hf^\sigma)}\}$ is weakly relatively  compact in $L^1([0,T]\times\Do)$.
\end{corollary}

\begin{proof}
Since $0\le \frac{Q^{-}(f^\sigma,\hf^\sigma)}{1+L(\hf^\sigma)}\le f^\sigma$, the weak compactness of $f^\sigma$ ensures the weak compactness of the loss term. The comparison inequality \eqref{bdd:Q-:Q+} ensures the 
    weak compactness of the gain term $\frac{Q^{+}(f^\sigma,\hf^\sigma)}{1+L(\hf^\sigma)}$.
\end{proof}

\section{Applications of velocity averaging}
\label{sec:weak}

In Section \ref{sec:weak-cmpt}, we have shown the weak relative compactness of 
\begin{align*}
\{f^\sigma\},\quad\{\hf^\sigma\},\quad\{g^{\sigma,\alpha}\}\quad\text{and}\quad\{\hg^{\sigma,\alpha}\}.
\end{align*}
Here, we fix subsequences such that
\begin{equation}
\label{conv:sec-3}
    f^\sigma\rightharpoonup f,\quad \hf^\sigma\rightharpoonup\hf, \quad g^{\sigma,\alpha}\rightharpoonup g^\alpha\quad \text{and}\quad  \hg^{\sigma,\alpha}\rightharpoonup \hg^\alpha.
\end{equation}
as $\sigma\to 0$ in $L^1([0,T]\times\Do)$.

In this section, we will show the following convergence results.
\begin{theorem}
    \label{thm:conv}
    Let $B$ be as in  Assumption~\ref{CK}.  
    \begin{enumerate}
        \item For any $\varphi\in L^\infty([0,T]\times\Do)$, as $\sigma\to0$, we have
        \begin{equation*}
          \int_{\R^d}\hf^{\sigma}\varphi \dd v\to \int_{\R^d}\hf\varphi\dd v\quad\text{and}\quad \int_{\R^d}f^{\sigma}\varphi\dd v\to \int_{\R^d}f\varphi\dd v
        \end{equation*}
        in $L^1([0,T]\times\R^d)$.

        \item For any $R\ge0$, as $\sigma\to0$, we have
        \begin{equation*} L(\hf^\sigma)\defeq\hf^\sigma*_v A \to L(\hf)\defeq \hf*_v A
        \end{equation*}
        in $L^1([0,T]\times \R^d\times B^v_R)$.

         \item We have $\hf$ and $f\in C([0,T];L^1(\Do))$.

\item For any $\alpha>0$, as $\sigma\to0$, we have
        \begin{align*}
        &\frac{Q^\pm(f^\sigma,\hf^\sigma)}{1+\alpha L(\hf^\sigma)}\rightharpoonup\frac{Q^\pm(f,\hf)}{1+\alpha L(\hf)}\quad\text{in}\quad L^1([0,T]\times\Do)\quad\text{and}\\
        &\frac{f^\sigma \hf^\sigma B}{1+\alpha L(\hf^\sigma)}  \rightharpoonup\frac{f\hf B}{1+\alpha L(\hf)} \quad\text{in}\quad
        L^1([0,T]\times \R^{3d}\times S^{d-1}).
        \end{align*}
       
    \end{enumerate}
\end{theorem}

We first recall the following averaging lemma.
\begin{theorem}[Velocity averages, \cite{DL89b,AC90}]
\label{thm:ave}
Let the pair $(h^n,H^n)$ be a weak solution of
\begin{equation*}
    \d_t h^n+v\cdot \nabla_x h^n=H^n\quad n\in\N_+,
\end{equation*}
where $\{h^n\}$ is weakly relative compact in $L^1([0,T]\times\Do)$, and for each compact set $K\subset (0,T)\times\R^d$,  $\{H^n\}$ is weakly relative compact set in $L^1(K)$. 
\begin{enumerate}[(i)]
   
    \item   For all  $\varphi\in L^\infty([0,T]\times\Do)$,$\int_{\R^d}h^n\varphi\dd v$ is relatively compact set in $L^1([0,T]\times \R^d)$.
    \item Let $(E,\mu)$ be an arbitrary measure space. For all $\varphi\in L^\infty([0,T]\times\Do;L^1(E)) $, we have $\int_{\R^d}h^n\varphi\dd v$ belong to a relatively compact set in $L^1([0,T]\times\R^d\times E)$.
\end{enumerate}
    
\end{theorem}

One can also apply a refined averaging lemma, where the weak compactness of $\{H_n\}$ can be weakened to $H_n\in L^1_\loc(\Do)$, see \cite{DLM91,GLPS88}.

We recall the equation \eqref{FBE:tilde}
\begin{equation}
 \label{TP:g-sigma-alpha}   (\d_t+v\cdot\nabla_x)\hg^{\sigma,\alpha}=\hQ^{\alpha}(f^\sigma,\hf^\sigma).
\end{equation}
In Lemma~\ref{lemma:f} and Lemma~\ref{lemma:Q}, we have shown that for any fixed $\alpha\in(0,1)$, $\{\hg^{\sigma,\alpha}\}$ and $\{\hQ^{\alpha,\pm}(f^\sigma,\hf^\sigma)\}$ are weakly relatively compact in $L^1([0,T]\times \Do)$ and $L^1([0,T]\times \R^d\times B^v_R)$. Let $\hg^\alpha$ be the weak limit of $g^{\sigma,\alpha}$ in $L^1([0,T]\times\Do)$. Then Theorem~\ref{thm:ave}-$(i)$ imlies that 
\begin{equation}
\label{conv:fix:alpha}
\int_{\R^d}\hg^{\sigma,\alpha}\varphi\dd v\to \int_{\R^d}\hg^\alpha\varphi\dd v\quad\text{in}\quad L^1([0,T]\times\R^d)
\end{equation}
as $\sigma\to0$
for all $\varphi\in L^\infty([0,T]\times \Do)$.

We recall that $\hf^{\sigma}(x,v)\defeq f^\sigma*_x\kappa^\sigma$ and $\hf^\sigma\rightharpoonup \hf$ in $L^1$.
Then based on the convergence \eqref{conv:fix:alpha}, we show Theorem~\ref{thm:conv}-$(1)$.
\begin{lemma}\label{ave:conv:sigma}
As $\sigma\to0$, we have 
\begin{align*}
\int_{\R^d}\hf^{\sigma}\varphi \dd v\to \int_{\R^d}\hf\varphi\dd v\quad\text{and}\quad \int_{\R^d}f^{\sigma}\varphi\dd v\to \int_{\R^d}f\varphi\dd v  
\end{align*}
in $L^1([0,T]\times \R^d)$ for all $\varphi\in L^\infty([0,T]\times \Do)$.
\end{lemma}
\begin{proof}
We first show the convergence corresponding to $\hf^\sigma$.
By the triangle inequality, we have
\begin{align*}
&\Big|\int_{\R^d}\hf^{\sigma}\varphi\dd v -\int_{\R^d}\hf\varphi\dd v \Big|\\
&\le \int_{\R^d}|\hf^{\sigma}-\hg^{\sigma,\alpha}|\varphi\dd v+ \Big|\int_{\R^d}(\hg^{\sigma,\alpha}-\hg^\alpha)\varphi\dd v\Big|+ \Big|\int_{\R^d}(\hg^\alpha-\hf)\varphi\dd v\Big|.
\end{align*}
As a consequence of \eqref{conv:fix:alpha}, we only need to show 
  \begin{equation}
  \label{conv:sigma:alpha}
\lim_{\alpha\to0}\sup_{t,\sigma}\|\hf^{\sigma}-\hg^{\sigma,\alpha}\|_{L^1(\Do)}=0.
  \end{equation}
Indeed, we have  $\hf^\sigma\rightharpoonup \hf$ and $\hg^{\sigma,\alpha}\rightharpoonup \hg^\alpha$ in $L^1([0,T]\times\Do)$. The uniform convergence \eqref{conv:sigma:alpha} ensures that $\lim_{\alpha\to0}\|\hg^\alpha-\hf\|_{L^1([0,T]\times\Do)}=0$. 
  
To show the uniform convergence \eqref{conv:sigma:alpha}, we first observe that $x\mapsto \alpha^{-1}\log(1+\alpha x)$ is concave, and for any $R\ge0$,  we have $0\le x-\alpha^{-1}\log(1+\alpha x)\le \varepsilon_\alpha(R)+x\mathbb{1}_{\{x\ge R\}}$, where $\varepsilon_\alpha(R)=R-\alpha^{-1}\log(1+\alpha R)$ and $\varepsilon_\alpha(R)\to0$ as $\alpha\to0$ for any fixed $R\ge 0$. Then we have
\begin{equation*}
\sup_{t,\sigma}\int_{\Do} |\alpha^{-1}\log(1+\alpha f^\sigma)-f^{\sigma}|\dd x\dd v\le \varepsilon_\alpha(R)+\sup_{t,\sigma}\int_{\Do}f^\sigma\mathbb{1}_{\{f^\sigma\ge R\}}\dd x\dd v,
  \end{equation*}
where we use
$\|f^\sigma_t\|_{L^1(\Do)}=\|f_0\|_{L^1(\Do)}$ for all $t\in[0,T]$ and $\sigma>0$. The uniform equi-integrability of $f^\sigma$ \eqref{equi-int:f-sigma} ensures that
\begin{equation*}
\lim_{R\to\infty}\sup_{t,\sigma}\int_{\Do}f^\sigma\mathbb{1}_{\{f^\sigma\ge R\}}\dd x\dd v=0.
  \end{equation*}
  Then we conclude $\|\alpha^{-1}\log(1+\alpha f^\sigma)-f^\sigma\|_{L^1(\Do)}\to0$ uniformly in $t$ and $\sigma$.
Moreover, by definition of convolution, we have 
\begin{align*}
    &\|\hg^{\sigma,\alpha}-\hf^\sigma\|_{L^1(\Do)}\\
    =&{}\|(\alpha^{-1}\log(1+\alpha f^\sigma)-f^{\sigma})*_x\kappa^\sigma\|_{L^1(\Do)}\\
     \le&{} \sup_{t,\sigma}\|\alpha^{-1}\log(1+\alpha f^\sigma)-f^{\sigma}\|_{L^1(\Do)}\|\kappa^\sigma\|_{L^1(\R^d)}\to0
\end{align*}
as $\alpha\to0$ uniformly in $\sigma$.

The convergence \eqref{conv:sigma:alpha} corresponding to $f^\sigma$ follows analogously. We apply Theorem~\ref{thm:ave}-$(i)$ on the equation~\eqref{FBE:renorm} for any fixed $\alpha\in(0,1)$
\begin{equation}
 \label{TP:g-sigma-alpha:f}   (\d_t+v\cdot\nabla_x)g^{\sigma,\alpha}=Q^{\alpha}(f^\sigma,\hf^\sigma)
\end{equation}
to derive
\begin{align*}
   \lim_{\sigma\to0} \int_{\R^d}g^{\sigma,\alpha}\varphi\dd v= \int_{\R^d}g^\alpha\varphi\dd v\quad\forall \varphi \in L^\infty([0,T]\times\Do)
\end{align*}
in $L^1([0,T]\times\R^d)$.
Similar to \eqref{conv:sigma:alpha}, we have
\begin{equation*}
  \label{conv:sigma:alpha:f}
\lim_{\alpha\to0}\sup_{t,\sigma}\|f^{\sigma}-g^{\sigma,\alpha}\|_{L^1(\Do)}=0.
  \end{equation*}   
\end{proof}

We show Theorem \ref{thm:conv}-$(2)$ by applying Theorem~\ref{thm:ave}-$(ii)$ with $\varphi=A(v-v_*)$ to the equation \eqref{TP:g-sigma-alpha}
\begin{equation}
 \label{TP:g-sigma-alpha-2}   (\d_t+v\cdot\nabla_x)\hg^{\sigma,\alpha}=\hQ^{\alpha}(f^\sigma,\hf^\sigma)
\end{equation}

\begin{lemma}
\label{ave:conv:L}
We recall the definition $L(\hf^\sigma)=\hf^\sigma*_vA$.
For any $R\ge0$, as $\sigma\to0$, we have
\begin{equation}
\label{conv:L}
L(\hf^{\sigma})\to L(\hf)\quad\text{in}\quad L^1([0,T]\times \R^d\times B_R^v).
\end{equation}

\end{lemma}
\begin{proof}
We first approximate the collision kernel $A$ by $A_k(v)=A(v)\mathbb{1}_{\{|v|\le k\}}\in L^1\cap L^\infty(\R^d)$. We define $L_k(g)=g*_v A_k$.
We apply Theorem~\ref{thm:ave}-$(ii)$ to equation \eqref{TP:g-sigma-alpha-2}  with $\varphi=A_k(v-v_*)$ to derive as $\sigma\to0$
\begin{equation}
\label{conv:L:1}
    L_k(\hg^{\sigma,\alpha})\to L_k(g^\alpha)\quad \text{in}\quad L^1([0,T]\times \Do).
\end{equation}

By using of the uniform convergence \eqref{conv:sigma:alpha}, we have, as $\alpha\to0$
\begin{align*}
    &\sup_{t,\sigma}\|(\hg^{\sigma,\alpha}-\hf^\sigma)*_vA_k\|_{L^1(\Do)} \le \sup_{t,\sigma}\|\hg^{\sigma,\alpha}-\hf^\sigma\|_{L^1(\Do)}\|A_k\|_{L^1(\R^d)}\to0.
\end{align*}
Thus, as $\alpha\to0$, we have
\begin{equation}
\label{conv:L:2}
 \sup_{t,\sigma}   L_k(\hg^{\sigma,\alpha})\to L_k(\hf^\sigma)\quad\text{and}\quad L_k(\hg^\alpha)\to L_k(\hf)
\end{equation}
in $L^1([0,T]\times\Do)$.  
The convergence \eqref{conv:L:1} and \eqref{conv:L:2} imply that as $\sigma\to0$
\begin{equation}
\label{conv:L:k}
L_k(\hf^{\sigma})\to L_k(\hf)\quad \text{in}\quad L^1((0,T)\times \Do).
\end{equation}

We note that for any $R\ge0$ $L(\hf^\sigma)\in L^\infty([0,T];L^1(\R^d\times B^v_R))$ and as in \eqref{k:ge:R}
\begin{equation*}
\lim_{k\to0}\sup_{t,\sigma}\|L_k(\hf^\sigma)-L(\hf^\sigma)\|_{L^1(\R^d\times B^v_R)}\to0.
\end{equation*}
We pass to the limit by letting $k\to\infty$ in \eqref{conv:L:k} to obtain \eqref{conv:L}. 
\end{proof}

We show Theorem~\ref{thm:conv}-$(3)$.
Based on the convergence of $\int_{\R^d} f^\sigma\varphi\dd v$ and  $L(\hf^\sigma)$ in $L^1$, we have the following lemma.
\begin{lemma}
   We have $\hf$ and $f\in C([0,T];L^1(\Do))$.
\end{lemma}
\begin{proof}
We first show $\hf\in C([0,T];L^1(\Do))$.  We recall the equation \eqref{TP:g-sigma-alpha} on the characteristic lines 
\begin{equation*}
    \d_t \hg^{\sigma,\alpha,\sharp}=\hQ^{\alpha}(f^\sigma,\hf^\sigma)^\sharp,
\end{equation*}
where 
\begin{align*}
&\hg^{\sigma,\alpha,\sharp}\defeq \hg^{\sigma,\alpha}(t,x+tv,v)\quad\text{and}\quad   \hQ^{\alpha}(f^\sigma,\hf^\sigma)^\sharp\defeq  \hQ^{\alpha}(f^\sigma,\hf^\sigma)(t,x+tv,v). 
\end{align*}

Without loss of generality, we assume $T=1$. We observe that on the bounded ball $B^{x,v}_R\subset \Do$
\begin{equation*}
    \|\hg^{\sigma,\alpha,\sharp}(t)-\hg^{\sigma,\alpha,\sharp}(s)\|_{L^1(B^{x,v}_R)}\le \int_s^t\int_{B^{x,v}_{2R}} \hQ^{\alpha}(f^\sigma,\hf^\sigma)\dd x\dd v\dd\tau.
\end{equation*}
For any fixed $\alpha\in(0,1)$, the weak compactness of $\{\hQ^{\alpha}(f^\sigma,\hf^\sigma)\}$ in $L^1([0,T]\times\R^d\times B^v_R)$ ensures the equi-continuity of $\hg^{\sigma,\alpha,\sharp}$. Indeed, the loss term can be bounded by 
\begin{equation}
\label{ineq:loss}
\begin{aligned}
&\int_s^t\int_{B^{x,v}_{2R}}\frac{Q^{-}(f^\sigma,\hf^\sigma)}{1+\alpha f^\sigma}
    *_x \kappa^\sigma\dd x\dd v\dd\tau\\
    =&{}\int_s^t\int_{B^{x,v}_{2R}}\frac{f^\sigma}{1+\alpha f^\sigma}(f^\sigma*_x\kappa^\sigma*_vA)\dd x\dd v\dd\tau\\
\le&{}\alpha^{-1}\int_{\Do}\hf^\sigma(x,v_*)\big(\int_{B_R^v}A(v-v_*)\dd v\big)\dd x \dd v_*\\
\lesssim &{}|t-s|\alpha^{-1}R^{\mu+d}\|f^\sigma\langle v\rangle^\mu\|_{L^1(\Do)},
\end{aligned}    
\end{equation}
where $\|f^\sigma\langle v\rangle^\mu\|_{L^1(\Do)}$ is uniformly bounded.
We use the inequality \eqref{relation:f':f} to compare $Q^+$ and $Q^-$, the gain term can be bounded by
\begin{equation}
    \label{ineq:gain}
    \begin{aligned}
&\int_s^t\int_{B^{x,v}_{2R}}\frac{Q^{+}(f^\sigma,\hf^\sigma)}{1+\alpha f^\sigma}
    *_x \kappa^\sigma\dd x\dd v\dd \tau\\
     \le&{} C\int_s^t\int_{B^{x,v}_{2R}}\frac{Q^{-}(f^\sigma,\hf^\sigma)}{1+\alpha f^\sigma}*_x\kappa^\sigma\dd x\dd v\dd\tau+(\log C)^{-1}\int_0^T\cD(f^\sigma)\dd t.
\end{aligned}
\end{equation}
Notice that $\cD(f^\sigma)\in L^1([0,T])$ as in Proposition~\ref{prop:moment:ent}. By choosing $C=|s-t|^{-\frac12}$, we have
\begin{equation}
\label{ineq:hg:s-t-2}
   \sup_\sigma \|\hg^{\sigma,\alpha,\sharp}(t)-\hg^{\sigma,\alpha,\sharp}(s)\|_{L^1(B^{x,v}_{2R})}\lesssim_{\alpha} |t-s|^{\frac12}+\big|\log|t-s|\big|.
\end{equation}
We combine \eqref{ineq:hg:s-t-2} and the uniform convergence \eqref{conv:sigma:alpha} $\lim_{\alpha\to0}\sup_{t,\sigma}\|\hg^{\sigma,\alpha}-\hf^\sigma\|_{L^1(\Do)}=0$
to derive the equi-continuity of $\{\hf^{\sigma\sharp}\}_{\sigma}$.
The weak convergence  $\hf^\sigma\rightharpoonup \hf$ in $L^1(\Do)$ uniformly in $t$ implies 
\begin{equation*}
    \sup_t \|\hf_t\|_{L^1((B^{x,v}_R)^c)}\to0\quad \text{as}\quad R\to\infty.
\end{equation*}
Hence, we have $\hf\in C([0,T];L^1(\Do))$.

We repeat the above arguments on the equation \eqref{TP:g-sigma-alpha:f} 
\begin{equation*}
    \d_t g^{\sigma,\alpha,\sharp}=(Q^{\alpha}(f^\sigma,\hf^\sigma))^\sharp
\end{equation*}
to obtain $f\in C([0,T];L^1(\Do))$.
\end{proof}

We show Theorem~\ref{thm:conv}-$(4)$.
\begin{lemma}
\label{lemma:inner:Q-}
As $\sigma\to0$, we have  
\begin{align}
\frac{ Q^\pm(f^\sigma,\hf^\sigma)}{1+L(\hf^\sigma)}\rightharpoonup
\frac{Q^\pm(f,\hf)}{1+L(\hf)}\quad\text{in}\quad L^1([0,T]\times\Do)\label{conv:Qpm-L},
\end{align}
and
\begin{align}
\frac{ f^\sigma\hf^\sigma_*}{1+L(\hf^\sigma)}B\rightharpoonup
\frac{f\hf_*}{1+L(\hf)}B\quad\text{in}\quad L^1([0,T]\times\R^{3d}\times S^{d-1}). \label{Sec8:A}
\end{align}
\end{lemma}

\begin{proof}
Since $\frac{L(\hf^\sigma)}{1+L(\hf^\sigma)}$ is bounded and converges to $ \frac{L(\hf)}{1+L(\hf)}$ in measure, the convergence \eqref{conv:Qpm-L} holds with $Q^-$. The comparison inequality \eqref{bdd:Q-:Q+} ensures the convergence of the gain term $Q^+$.

We show \eqref{Sec8:A}. We  first approximate the collision kernel $B$  by $B_k(v,\omega)=B(v,\omega){}\mathbb{1}_{\{|v|\le R\}}\in L^\infty\cap L^1(\R^d\times S^{d-1})$. For any $\varphi\in L^\infty([0,T]\times\R^{3d}\times S^{d-1})$, Theorem~\ref{thm:ave}-$(ii)$ implies that
\begin{equation}
\label{con:B-k}
   \int_{\R^d\times S^{d-1}}\hf^\sigma_* \varphi B_k(v-v_*) \dd v_*\dd\omega\to \int_{\R^d\times S^{d-1}}\hf_* \varphi B_k(v-v_*)\dd v_*\dd\omega
\end{equation}
in $L^1([0,T]\times \Do)$.
Combining with $L(\hf^\sigma)
\to L(\hf)$ in $L^1([0,T]\times \R^d\times B^v_R)$ for any $R\ge 0$ showed in  Lemma~\ref{ave:conv:L}, we have 
\begin{align*}
   \frac{\int_{\R^d\times S^{d-1}}\hf^\sigma_* \varphi B_k \dd v_*\dd\omega}{1+L(\hf^\sigma)}\to \frac{\int_{\R^d\times S^{d-1}}\hf_* \varphi B_k\dd v_*\dd\omega}{1+L(\hf)}
\end{align*}
almost everywhere on $[0,T]\times \Do$.
Moreover, the following uniform bound holds
\begin{align*}
 \left|\frac{\int_{\R^d\times S^{d-1}}\hf^\sigma_* \varphi B_k \dd v_*\dd\omega}{1+L(\hf^\sigma)}\right|\le \|\varphi\|_{L^\infty}\frac{L(\hf^\sigma)}{1+L(\hf^\sigma)}\le \|\varphi\|_{L^\infty}.
\end{align*}
Combining with $f^\sigma\rightharpoonup f$ in $L^1([0,T]\times\Do)$, the convergence \eqref{Sec8:A} holds with $B=B_k$ . 

We pass to the limit by letting $k\to\infty$. We only need to show the convergence \eqref{con:B-k} holds for $B$ in $L^1([0,T]\times\R^d\times B^v_R)$ for any $R\ge0$. Without loss of generality, we assume the positivity of $\varphi$ and $k\ge R$. Indeed, we have
\begin{align*}
    &\int_{\R^{2d}\times B^v_R\times S^{d-1}}\hf^\sigma_*\varphi B(v-v_*)\mathbb{1}_{\{|v-v_*|\ge k\}}\dd x\dd v_* \dd v\dd\omega\\
    \lesssim&{} \|\varphi\|_{L^\infty}\int_{\R^d\times (B^v_{k-R})^c}\hf^\sigma \langle v\rangle^\mu\dd x\dd v\to0\quad\text{as}\quad k\to\infty
\end{align*}
uniformly in $\sigma$, since $\|\hf^\sigma\langle v\rangle ^2\|_{L^1(\Do)}$ is uniformly bounded.
        
\end{proof}

If the collision kernel is bounded and integrable, then we have the following convergence results.
\begin{lemma}
\label{conv:Q:compo}
Let $B\in L^1\cap L^\infty(\R^d\times S^{d-1})$. Let $(E,\mu)$ be a measurable space.
Let $k:\R_+\to\R_+$ be a bounded, Lipschitz function, $k(0)=0$ and $|k'(z)|\lesssim (1+z)^{-1}$. Then we have 
    \begin{equation}
    \label{conv:Q:beta}
        \int_{\R^d}Q^{\pm}(k(f^\sigma),k(\hf^\sigma))\varphi\dd v\to \int_{\R^d}Q^{\pm}(\tilde k,\tilde{\mathbb{k}})\varphi\dd v \quad\forall \varphi\in L^\infty([0,T]\times\Do;L^1(E))
    \end{equation}
   as $\sigma\to0$ in $L^1([0,T]\times \R^d\times E)$, where $\tilde k$ and $\tilde{\mathbb{k}}$ denote the weak limits of $k(f^\sigma)$ and $k(\hf^\sigma)$ in $L^1([0,T]\times \Do)$.
\end{lemma}
\begin{proof}
The assumptions on $k$ imply that
\begin{equation*}
    0\le k(f^\sigma)\le f^\sigma\quad\text{and}\quad 0\le k(\hf^\sigma)\le \hf^\sigma.
\end{equation*}
The weak compactness of $k(f^\sigma)$ and $k(\hf^\sigma)$ in $L^1([0,T]\times \Do)$ follows from the weak compactness of $f^\sigma$ and $\hf^\sigma$.

We apply velocity averaging Theorem~\ref{thm:ave}-$(ii)$ to equation
\begin{equation*}
    (\d_t+v\cdot\nabla_x)k(f^\sigma)=k'(f^\sigma)Q(f^\sigma,\hf^\sigma),
\end{equation*}
where $k'(f^\sigma)Q(f^\sigma,\hf^\sigma)\in L^1([0,T]\times\Do)$ and weakly relatively compact, analogous to the weak compactness $Q^{\alpha}(f^\sigma,\hf^\sigma)$ in Lemma~\ref{lemma:Q}. Thus we have
\begin{equation*}
\label{conv:beta:1}
    \int_{\R^d}k(f^\sigma)\varphi\dd v\to \int_{\R^d}\tilde k\varphi\dd v\quad \text{in }^1([0,T]\times\R^d\times E)
\end{equation*}
for all $\varphi\in L^\infty([0,T]\times\Do;L^1(E))$.

For any fixed $\alpha\in(0,1)$, we consider the equation
\begin{equation*}
(\d_t+v\cdot\nabla_x)k(\hg^{\sigma,\alpha})=k'(\hg^{\sigma,\alpha})\hQ^{\alpha}(f^\sigma,\hf^\sigma),
\end{equation*}
where $k(\hg^{\sigma,\alpha})$ and $k'(\hg^{\sigma,\alpha})\hQ^{\alpha}(f^\sigma,\hf^\sigma)$ are weakly relatively compact in $L^1([0,T]\times\Do)$ since  $0\le k(\hg^{\sigma,\alpha})\le \hg^{\sigma,\alpha}$ and $0\le k'(\hg^{\sigma,\alpha})\hQ^{\alpha}(f^\sigma,\hf^\sigma)\le \hQ^{\alpha}(f^\sigma,\hf^\sigma)$. Theorem~\ref{thm:ave}-$(ii)$ ensures
that $\{\int_{\R^d}k(\hg^{\sigma,\alpha})\varphi\dd v\}$ is a Cauchy sequence in
$L^1([0,T]\times\R^d\times E)$. 

Since $k$ is a Lipschitz function, by using of the uniform convergence \eqref{conv:sigma:alpha}, we have
\begin{equation*}
\lim_{\alpha\to0}\sup_{t,\sigma}\|k(\hf^\sigma)-k(\hg^{\sigma,\alpha})\|_{L^1(\Do)}=0.
\end{equation*}
By assumption $A=\int_{S^{d-1}}B\dd\omega\in L^1\cap L^\infty(\R^d)$, velocity averaging Theorem~\ref{thm:ave}-$(ii)$ implies that
\begin{equation}
\label{conv:beta:2}
L(k(\hf^\sigma))\to L(\tilde{\mathbb{k}})\quad\text{in}\quad L^1([0,T]\times\Do).
\end{equation}

We are ready to show the convergence \eqref{conv:Q:beta} for loss term $Q^-(k(f^\sigma),k(\hf^\sigma))$
\begin{align*}
    &\Big\|\int_{\R^d}k(f^\sigma) L(k(\hf^\sigma))\varphi\dd v-\int_{\R^d}\tilde k L(\tilde{\mathbb{k}})\varphi\dd v\Big\|_{L^1([0,T]\times \R^d\times E)}\\
    \le&{} \Big\|\int_{\R^d}\big(k(f^\sigma) -\tilde k\big)L(\tilde{\mathbb{k}})\varphi\dd v\Big\|_{L^1}+\Big\|\int_{\R^d} k(f^\sigma)\big(L(k(\hf^\sigma)-L(\tilde {\mathbb{k}}))\varphi\dd v\Big\|_{L^1}\\
    \le&{} \Big\|\int_{\R^d}\big(k(f^\sigma) -\tilde k\big)L(\tilde{\mathbb{k}})\varphi\dd v\Big\|_{L^1([0,T]\times \R^d\times E)}\\
    &+C\|\varphi\|_{L^\infty([0,T]\times \Do; L^1(E))}\| L(k(\hf^\sigma)-L(\tilde {\mathbb{k}})\|_{L^1([0,T]\times \Do)}\\
    \to&{}0\quad\text{as }\sigma\to0,
\end{align*}
where $L(\tilde {\mathbb{k}}) \varphi \in L^\infty([0,T]\times \Do;L^1(E))$ since $|L(\tilde{\mathbb{k}})|\le\|k\|_{L^\infty}\|A\|_{L^1}$.

Concerning the convergence of the gain term $Q^+(k(f^\sigma),k(\hf^\sigma))$ in \eqref{conv:Q:beta}, changing of variables $(v,v_*)\mapsto (v',v_*')$ implies that
\begin{align*}
\int_{\R^d}Q^+(k(f^\sigma),k(\hf^\sigma))\varphi\dd v=\int_{\Do}  k(f^\sigma)k(\hf^\sigma)_*\Gamma \dd v\dd v_*, 
\end{align*}
where $k(\hf^\sigma)_*=k(\hf^\sigma)(x,v_*)$ and $\Gamma$ is defined as
\begin{align*}
    \Gamma=\int_{S^{d-1}}B(v-v_*,\omega)\varphi(x,v-\langle v-v_*,\omega\rangle \omega)\dd\omega.
\end{align*}
Notice that $\Gamma\in L^\infty([0,T]\times\R^{3d};L^1(E))\cap L^\infty([0,T]\times\Do;L^1(\R^d_{v}\times E))$. Similar to \eqref{conv:beta:2}, we have 
\begin{align*}
    \int_{\R^d}k(\hf^\sigma)_*\Gamma\dd v_*\to \int_{\R^d}\tilde{\mathbb{k}}_*\Gamma\dd v_*\quad\text{in}\quad L^1([0,T]\times\Do\times E).
\end{align*}
Then the convergence of the gain term follows analogously to the convergence of the loss term.

\end{proof}

\section{Strong convergence results}
\label{sec:strong}

We will show that the gain term $Q^+(f^\sigma,\hf^\sigma)$ converges in measure in Subsection~\ref{sub-sec:gain} by using the smoothing effect of the operator $Q^+$. Then we will apply a general convergence result by \textcite{Lio94} to conclude that (up to a subsequence)
\begin{equation*}
    f^\sigma\to f\quad\text{in}\quad C([0,T];L^1(\Do))
\end{equation*}
and $f$ is a renormalised solution of the inhomogeneous Boltzmann equation \eqref{IBE}.

In this section, we assume that the collision kernel $B$ satisfies Assumption~\ref{CK}.

\subsection{Convergence in measure}
\label{sub-sec:gain}

\begin{theorem}
\label{thm:conv:measure}
We have
\begin{equation}
\label{conv:Q+:meas}
   Q^{+}(f^\sigma,\hf^\sigma)\to Q^+(f,\hf) \quad\text{in measure}
\end{equation}
 on $[0,T]\times \R^d\times B^{v}_R$ for all $R>0$.
\end{theorem}

For any $M>0$, we define the regularised sequence $\{f^\sigma_M\}$ and $\{\hf^\sigma_M\}$ in the following way
\begin{align*}
    &f^\sigma_M\defeq (f^\sigma\wedge M)\varphi_M(v)\varphi_M(x)\quad\text{and}\quad \hf^\sigma_M\defeq (\hf^\sigma\wedge M)\varphi_M(v)\varphi_M(x),
\end{align*}
where $\varphi\in C_c^\infty(\R^d)$, $0\le \varphi\le1$, $\varphi=1$ on $B_1(0)$, $\varphi=0$ on $B_2(0)$, and $\varphi_M\defeq \varphi(\frac{\cdot}{M})$. 

By definition, we have 
\begin{align*}
    &0\le f^\sigma_M,\,\hf^\sigma_M\le M,\quad \supp(f^\sigma_M,\,\hf^\sigma_M)\subset [0,T]\times B^{x,v}_{2M}.
\end{align*}

To show Theorem~\ref{thm:conv:measure}, we will prove the following statements in Subsection~\ref{sub-sec:proof-1} and Subsection~\ref{sub-sec:proof-2}.

\begin{enumerate}

    \item We show that 
    \begin{equation}
    \label{conv:f-sigma-M:M}
        \frac{Q^+( f^\sigma_M,\hf^\sigma_M)}{1+L(\hf^\sigma)}\to\frac{Q^{+}( f^\sigma,\hf^\sigma)}{1+L(\hf^\sigma)}\quad\text{in}\quad L^1([0,T]\times\Do)
    \end{equation}
    as $M\to\infty$ uniformly in $\sigma\in(0,1)$. 
    \item  Since $0\le f^\sigma_M\le f^\sigma$ and  $0\le \hf^\sigma_M\le \hf^\sigma$, we denote $\tilde f_M$ and  $\tilde \hf_M$ by
    the weak limits of $f^\sigma_M$ and $\hf^\sigma_M$ in $L^1([0,T]\times\Do)$.
   For any fixed $M>0$, we show
    \begin{equation}
    \label{conv:f-sigma-M:sigma}
        Q^+(f^\sigma_M,\hf^\sigma_M)\to Q^+(\tilde f_M,\tilde\hf_M)\quad\text{in}\quad L^1([0,T]\times \Do)
    \end{equation}
    as $\sigma\to0$ for all $R>0$. 
\end{enumerate}

Now we are ready to show \eqref{conv:Q+:meas}. 

For a fixed $M>0$. The convergence  \eqref{conv:f-sigma-M:sigma} implies that
\begin{equation}
    \label{conv:f-sigma-12-M}
    \Big\{\frac{Q^+( f^{\sigma}_M,\hf^{\sigma}_M)}{1+L(\hf^{\sigma})}\Big\}\quad\text{in a Cauchy sequence in }{L^1([0,T]\times \R^d\times B^v_R},
\end{equation}
where we use $L(\hf^{\sigma})\to L(\hf)$ in $L^1([0,T]\times \R^d\times B^v_R)$ as in Lemma~\ref{ave:conv:L}, and 
\begin{equation*}
\begin{aligned}
&\Big\|\frac{Q^+(f^{\sigma_1}_M,\hf^{\sigma_1}_M)}{1+L(\hf^{\sigma_1})}-\frac{Q^+(f^{\sigma_2}_M,\hf^{\sigma_2}_M)}{1+L(\hf^{\sigma_2})}\Big\|_{L^1([0,T]\times \R^d\times B^{v}_R)} \\
\le&{}\Big\|\frac{Q^+( f^{\sigma_1}_M,\hf^{\sigma_1}_M)}{1+L(\hf^{\sigma_1})}-\frac{Q^+( f^{\sigma_2}_M,\hf^{\sigma_2}_M)}{1+L(\hf^{\sigma_1})}\Big\|_{L^1}+\Big\|\frac{Q^+(f^{\sigma_2}_M,\hf^{\sigma_2}_M)}{1+L(\hf^{\sigma_1})}-\frac{Q^+(f^{\sigma_2}_M,\hf^{\sigma_2}_M)}{1+L(\hf^{\sigma_2})}\Big\|_{L^1}\\
\le&{}\|Q^+(f^{\sigma_1}_M,\hf^{\sigma_1}_M)-Q^+(f^{\sigma_2}_M,\hf^{\sigma_2}_M)\|_{L^1}+\underbrace{\|Q^+(f^{\sigma_2}_M,\hf^{\sigma_2}_M)\|_{L^\infty}}_{\le M^2\|A\|_{L^\infty(B_{2M})}}\|L(\hf^{\sigma_1})-L(\hf^{\sigma_2})\|_{L^1}\\
\to&{}0\quad\text{as }\sigma_1,\,\sigma_2\to0.
\end{aligned}
\end{equation*}

We combine \eqref{conv:f-sigma-M:M} and \eqref{conv:f-sigma-12-M} to derive 
\begin{align*}
&\Big\{\frac{Q^+(f^{\sigma},\hf^{\sigma})}{1+L(\hf^{\sigma})}\Big\}\quad\text{is a Cauchy sequence in 
}L^1([0,T]\times \R^d\times B^{v}_R).
\end{align*}

We recall from Lemma~\ref{lemma:inner:Q-} that
\begin{align*}
    \frac{Q^{\pm}(f^\sigma,\hf^\sigma)}{1+L(\hf^\sigma)}\rightharpoonup  \frac{Q^{\pm}(f,\hf)}{1+L(\hf)}\quad\text{in}\quad L^1([0,T]\times\Do). 
\end{align*}

 Then  the uniqueness of the weak and strong limits implies
\begin{equation}
\label{strong:L1:sigma}
    \frac{Q^{+}(f^\sigma,\hf^\sigma)}{1+L(\hf^\sigma)}\to  \frac{Q^{+}(f,\hf)}{1+L(\hf)}\quad\text{in}\quad L^1([0,T]\times  \R^d\times B^{v}_R). 
\end{equation}

Moreover, since $L(f^\sigma)\to L(\hf)$ in $L^1([0,T]\times \R^d\times B^{v}_R)$, we derive the convergence \eqref{conv:Q+:meas} 
\begin{equation*}
    Q^{+}(f^\sigma,\hf^\sigma)\to Q^{+}(f,\hf)\quad\text{in measure on}\quad [0,T]\times  \R^d\times B^{v}_R. 
\end{equation*}

\subsubsection{Proof of  \eqref{conv:f-sigma-M:M}}
\label{sub-sec:proof-1}
We first observe 
 \begin{align*}
 &0\le f^\sigma - (f^\sigma\wedge M)\le f^\sigma\mathbb{1}_{\{f^\sigma>M\}}\quad\text{and}\quad 0\le \hf^\sigma - (\hf^\sigma\wedge M)\le \hf^\sigma\mathbb{1}_{\{\hf^\sigma>M\}},
 \end{align*}
which implies $0\le Q^+( f^\sigma_M,\hf^\sigma_M)\le Q^+(f^\sigma,\hf^\sigma)$.

By definition of $\varphi_M$, we have 
\begin{equation}
\label{diff:Q+:f-sigma-M}
\begin{aligned}
0
\le &{} \int_0^T\int_{\Do}\frac{Q^+( f^\sigma,\hf^\sigma)}{1+L(\hf^\sigma)}-\frac{Q^+(f^\sigma_M,\hf^\sigma_M)}{1+L(\hf^\sigma)}\dd x\dd v\dd t\\
\le &{}\int_0^T\int_{\Do}\frac{Q^+(f^\sigma,\hf^\sigma)}{1+L(\hf^\sigma)}\mathbb{1}_{\{|x|>M\}}\dd x\dd v\dd t\\
&+\int_0^T\int_{\R^{3d}\times S^{d-1}}\frac{( \hf^\sigma)'_*(f^\sigma)'}{1+L(\hf^\sigma)}(\mathbb{1}_{\{|v'|>M\}}+\mathbb{1}_{\{|v_*'|>M\}})B\dd x\dd v\dd v_*\dd t\\
&+\int_0^T\int_{\R^{3d}\times S^{d-1}}\frac{(\hf^\sigma)_*'(f^\sigma)'}{1+L(\hf^\sigma)}\Big(\mathbb{1}_{\{(\hf^\sigma)_*'+(f^\sigma)'>M\}}\Big)B\dd x\dd v\dd v_*\dd t,
\end{aligned}
\end{equation}
where $(\hf^\sigma)_*'=\hf^\sigma(x,v_*')$ and $(f^\sigma)'=f^\sigma(x,v')$. 

The weak compactness of $\frac{Q^+(f^\sigma,\hf^\sigma)}{1+L(\hf^\sigma)}$ ensures the smallness of the first term on the right-hand side of the inequality.

To estimate the last two terms on the right-hand side, we compare $(f^\sigma)'(\hf^\sigma)'_*$ and $L(\hf^\sigma)$. We first recall the comparison inequality \eqref{relation:f':f}, for any constant $C>0$, 
\begin{align*}
    &(\hf^\sigma)'_*(f^\sigma)'\le C f^\sigma\hf^\sigma_*\\
    &\quad+(\log C)^{-1}\Big(\big((f^\sigma)'_*(f^\sigma)'-f^\sigma_*f^\sigma\big)\log\frac{(f^\sigma)'_*(f^\sigma)'}{f^\sigma_*f^\sigma}\Big)*_x \kappa^\sigma,
\end{align*}
where we have the uniform bound
\begin{align*}
   \int_0^T\int_{\R^{4d}\times S^{d-1}}\kappa^\sigma B\big((f^\sigma)'_*(f^\sigma)'-f^\sigma_*f^\sigma\big)\log\frac{(f^\sigma)'_*(f^\sigma)'}{f^\sigma_*f^\sigma}=\int_0^T\cD(f^\sigma)\dd t\le C
\end{align*}
as in Proposition \ref{prop:moment:ent}. Since one can choose the constant $C>0$ arbitrarily large, to show the vanishing of \eqref{diff:Q+:f-sigma-M}, we only need to show
    \begin{align}
&\int_0^T\int_{\R^{3d}\times S^{d-1}}\frac{f^\sigma\hf^\sigma_*}{1+L(\hf^\sigma)}(\mathbb{1}_{\{|v'|>M\}}+\mathbb{1}_{\{|v_*'|>M\}})B\to0\quad\text{and}\label{con:rhs:sigma-M:1}\\
&\int_0^T\int_{\R^{3d}\times S^{d-1}}\frac{f^\sigma\hf^\sigma_*}{1+L(\hf^\sigma)}\Big(\mathbb{1}_{\{(\hf^\sigma)_*'+(f^\sigma)'>M\}}\Big)B\to0  \label{con:rhs:sigma-M:2}  
    \end{align}
as $M\to\infty$ uniformly in $\sigma\in(0,1)$.

We recall Lemma~\ref{lemma:inner:Q-}
that 
\begin{equation*}
    \frac{ f^\sigma\hf^\sigma_*}{1+L(\hf^\sigma)}B\rightharpoonup\frac{ f\hf_*}{1+L(\hf)}B\quad\text{in }L^1([0,T]\times\R^{3d}\times S^{d-1}).
\end{equation*}
We show \eqref{con:rhs:sigma-M:1} and \eqref{con:rhs:sigma-M:2} by the Dunford-Pettis theorem~\ref{thm:DP}. By  using of the energy conservation law $|v|^2+|v_*|^2=|v'|^2+|v_*'|^2$, we have 
\begin{equation*}
\label{Omega:decom}
\begin{aligned}
\{|v'|>M\}\cup\{|v'_*|>M\}\subset \{|v|>M/\sqrt 2\}\cup\{|v_*|>M/\sqrt 2\}.
\end{aligned}
\end{equation*}
Thus, we have the convergence \eqref{con:rhs:sigma-M:1} as $M\to+\infty$. 

As a consequence of \eqref{con:rhs:sigma-M:1}, to show \eqref{con:rhs:sigma-M:2} we only need to show for any fixed $R>0$, we have
\begin{align*}
    \int_0^T\int_{\{|v'|,|v_*'|\le R\}}\frac{f^\sigma\hf^\sigma_*}{1+L(\hf^\sigma)}\Big(\mathbb{1}_{\{(\hf^\sigma)_*'+(f^\sigma)'>M\}}\Big)B\to0  \end{align*}
    as $M\to+\infty$.
We note that
\begin{equation}
\label{Omega:decom-2}
\{|v'|\le R\}\cap\{|v'_*|\le R\}\subset \{|v|\le{\sqrt2}R\}\cap\{|v_*|\le{\sqrt2}R\}.
\end{equation}
By changing of variables $(v',v_*')\mapsto(v,v_*)$, we have
\begin{equation*}
\label{Sec8:B}
\begin{aligned}
&\int_0^T\int_{\R^{3d}\times S^{d-1}} \Big(\mathbb{1}_{\{(f^\sigma)'+(\hf^\sigma)'_*>M\}}\Big)\mathbb{1}_{\{|v'|,|v'_*|\le \sqrt2R\}}\\
    \le&{} \frac{T}{M}\int_{\R^{3d}\times S^{d-1}} (f^\sigma+\hf^\sigma_*)\mathbb{1}_{\{|v'|,|v'_*|\le \sqrt2R\}}\\
    \lesssim&{} \frac{R^d}{M}\to0\quad\text{as}\quad M\to\infty,
\end{aligned}
\end{equation*}
which implies \eqref{con:rhs:sigma-M:2}. 

Thus, the convergence \eqref{conv:f-sigma-M:M} holds.

\subsubsection{Proof of \eqref{conv:f-sigma-M:sigma}}
\label{sub-sec:proof-2}
    
   By definition $\supp(\hf^\sigma_M),\,\supp(f^\sigma_M)\subset [0,T]\times B^{x,v}_{2M}$. Combining with \eqref{Omega:decom-2}, we have 
    \begin{equation}
\supp\big((\hf^\sigma_M)'\big),\, \supp\big((f^\sigma_M)'\big)\subset[0,T]\times B_{2M}^x\times B^v_{2\sqrt2 M}\times S^{d-1}.
    \end{equation}
Since  $(f^\sigma_M)'(\hf^\sigma_M)_*'$ has compact supports in $v$ and $v_*$, without loss of generality, we assume the collision kernel $B$ has compact support and $B\in L^1\cap L^\infty(\R^d\times S^{d-1})$. We approximate $B$ by $B^\varepsilon\in C^\infty_c(\R^d\times S^{d-1})$ and $B_\varepsilon$ vanishes near $0$ such that
\begin{equation*}
    \|B-B^\varepsilon\|_{L^1(\R^d\times S^{d-1})}\to 0\quad\text{as} \quad \varepsilon\to0.
\end{equation*}

We recall that $\tilde f_M$ and  $\tilde \hf_M$ are 
    the weak limits of $f^\sigma_M$ and $\hf^\sigma_M$ in $L^1([0,T]\times\Do)$. To show \eqref{conv:f-sigma-M:sigma}, we only need to show
\begin{equation*}
Q^+_\varepsilon(f^\sigma_M,\hf^\sigma_M)\defeq \int_{\R^d\times S^{d-1}}(f^\sigma_M)'(\hf^\sigma_M)_*'B^\varepsilon\dd v_*\dd\omega\to Q^+_\varepsilon(\tilde f_M,\tilde\hf_M) 
\end{equation*}
in $L^1([0,T]\times \Do)$, since
\begin{align*}
    &\|Q^+(f^\sigma_M,\hf^\sigma_M)-Q^+_\varepsilon(f^\sigma_M,\hf^\sigma_M)\|_{L^1([0,T]\times\Do)}\le C(M)\|B-B^\varepsilon\|_{L^1(\R^d\times S^{d-1})}.
\end{align*}
Hence, we only need to show \eqref{conv:f-sigma-M:sigma} for collision kernels $B\in C_c^\infty(\R^d\times S^{d-1})$.

Since $(f^\sigma_M)'(\hf^\sigma_M)_*'$ is bounded with compact support, we have
    \begin{align*}
        \|Q^+(f^\sigma_M,\hf^\sigma_M)\|_{L^\infty([0,T];L^p(\Do))}\le C(M),\quad p=1,\infty .
    \end{align*}
  Let $\rho\in\C^\infty_c(\R^d;\R_+)$, $\int_{\R^d}\rho=1$ and $\rho^\delta\defeq \delta^{-d}\rho(\frac{\cdot}{\delta})$. We approximate  $Q^{+}(f^\sigma_M,\hf^\sigma_M)$ by
    \begin{align*}
     Q^{+}_\delta(f^\sigma_M,\hf^\sigma_M) \defeq Q^{+}(f^\sigma_M,\hf^\sigma_M) *_v \rho^\delta. 
    \end{align*}

As a consequence of velocity averaging theorem (Lemma~\ref{conv:Q:compo}),
for any fixed $\delta\in(0,1)$, the convergence \eqref{conv:f-sigma-M:sigma} holds with respect to $Q^{+}_\delta(f^\sigma_M,\hf^\sigma_M)$ i.e. for any fixed $M>0$,
    \begin{equation}
        \label{conv:M:sig}
 \begin{aligned}
 Q^{+}_\delta(f^\sigma_M,\hf^\sigma_M) \to Q^+_\delta(\tilde f_M,\tilde\hf_M)\quad \text{in}\quad L^1([0,T]\times \Do)
    \end{aligned}
      \end{equation}
      as $\sigma\to0$.
  
By definition, we have 
\begin{align}
    Q^{+}(f_M,\hf_M)*_v\rho^\delta\to Q^{+}(\tilde f_M,\tilde\hf_M)\quad\text{in}\quad L^p([0,T]\times\Do) \label{con:Q+-delta}
\end{align}  
as $\delta\to0$ for all $p\in[1,\infty)$.
Hence, to show 
\eqref{conv:f-sigma-M:sigma}, we only need to show 
\begin{align}
    Q^{+}(f^\sigma_M,\hf^\sigma_M) *_v \rho^\delta\to Q^{+}(f^\sigma_M,\hf^\sigma_M)\quad\text{in}\quad L^1([0,T]\times\Do)\label{conv:uni:delta}
\end{align}
as $\delta\to0$ uniformly in $\sigma\in(0,1)$, which is a direct consequence of the smoothing effect of $Q^+$.


    Indeed, since $Q^+(f^\sigma_M,\hf^\sigma_M)$ has compact support in $x$ and $v$, to show \eqref{conv:uni:delta} we only need to show that 
\begin{align*}
Q^+_\delta(f^\sigma_M,\hf^\sigma_M)\to  Q^{+}(f^\sigma_M,\hf^\sigma_M)\quad\text{in}\quad L^2([0,T]\times\Do)
\end{align*}  
as $\delta\to0$ uniformly in $\sigma$.
We define the Fourier transform in variable $v$
\begin{equation*}
  \cF( f)(\xi)\defeq \int_{\R^d}f(v)e^{-i\xi\cdot v}\dd v.  
 \end{equation*}
By using the Plancherel theorem,
 we only need to show
\begin{align*}
   &\|\cF_v Q^+_\delta(f^\sigma_M,\hf^\sigma_M)-\cF_v Q^{+}(f^\sigma_M,\hf^\sigma_M)\|^2_{ L^2([0,T]\times\Do)}\\
   =&\int_0^T\int_{\Do}\cF_v( Q^+(f^\sigma_M,\hf^\sigma_M))^2(1-\cF_v\rho(\delta \xi))^2\dd x\dd\xi \dd t\to0
\end{align*} 
as $\delta\to0$ uniformly in $\sigma$.
Since $\|\rho^\delta\|_{L^1(\R^d)}=1$, we only need to show that
\begin{equation}
\label{uni:vas}
  \int_0^T\int_{\Do}\cF_v( Q^+(f^\sigma_M,\hf^\sigma_M))^2\mathbb{1}_{\{|\xi|\ge R\}}\dd x\dd\xi \dd t\to0
\end{equation} 
as $R\to\infty$ uniformly in $\sigma$. 

We define the Sobolev spaces
    \begin{align*}
      H^{s}(\R^d)\defeq \{f\in \cD'(\R^d)\mid \langle \xi\rangle^s \cF(f)(\xi)\in L^2(\R^d)\}.
    \end{align*}
    If $Q^+(f^\sigma_M,\hf^\sigma_M)$ is uniformly bounded in $ {L^2((0,T)\times\R^d;H^{\frac{d-1}{2}}(\R^d))}$
\begin{equation}
    \label{SE}
\begin{aligned}
    &\|Q^+(f^\sigma_M,\hf^\sigma_M)\|^2_{L^2((0,T)\times\R^d;H^{\frac{d-1}{2}}(\R^d))}\\
    =&{}\int_0^T\int_{\Do}\cF_v(Q^+(f^\sigma_M,\hf^\sigma_M))^2\langle \xi\rangle^{\frac{d-1}{2}}\dd x\dd\xi \dd t\le C,
\end{aligned}
\end{equation}
where $C$ is independent of $\sigma$. Then we have the desired uniform convergence \eqref{uni:vas}
\begin{align*}
  \int_0^T\int_{\Do}\cF_v( Q^+( f^\sigma_M,\hf^\sigma_M))^2\mathbb{1}_{\{|\xi|\ge R\}}\dd x\dd\xi \dd t\le \frac{C}{1+R^{d-1}}\to0
\end{align*} 
uniformly in $\sigma$. Thus, we have the convergence \eqref{conv:uni:delta}.

Indeed, the following lemma of the smoothing effect of the gain term ensures the uniform Sobolev bound \eqref{SE}, where we treat $t$ and $x$ as parameters.
\begin{theorem}[\cite{Lio94}, Part I, Theorem 4.1]\label{smooth}
Let $B(v,w)\in C^\infty(\R^d\times S^{d-1})$  be a collision kernel such that $B(v,w)$ depends only on $|v|$ and $|\langle v,w\rangle|$, $B$ vanishes when $|v|$ near $0$ and $+\infty$ uniformly in $w$, and also vanishes when $|\langle v,w\rangle|$ near $0$ and $|v|$.
   We define
   \begin{equation*}
       Q^+(g,f)(v)\defeq \int_{\R^d\times S^{d-1}}B(v-v_*)f(v_*')g(v')\dd v_*\dd\omega.
   \end{equation*}
  For all $f\in L^1(\R^d)$ and $g\in L^2(\R^d)$, we have
   \begin{equation*}
  \|Q^+(g,f)\|_{H^{\frac{d-1}{2}}(\R^d)}\le C \|f\|_{L^1(\R^d)}\|g\|_{L^2(\R^d)}. 
   \end{equation*}
    \end{theorem}

We apply Theorem \ref{smooth} to $Q^+(f^\sigma_M,\hf^\sigma_M)$. Since $f^\sigma_M$ and $\hf^\sigma_M$ are bounded with compact support, we have the uniform bound  
\begin{align*}
    &\|Q^+(f^\sigma_M,\hf^\sigma_M)\|^2_{L^2((0,T)\times\R^d;H^{\frac{d-1}{2}}(\R^d))}\\
    \lesssim&{}\int_0^T\int_{\R^d}\|f^\sigma_M(v)\|_{L^1(\R^d)}^2\|\hf^\sigma_M(v)\|_{L^2(\R^d)}^2\dd x\dd t\\
    \le&{} C(M,d)\quad \forall \sigma\in(0,1).
\end{align*}

\subsection{Strong convergence}
We will apply the following general convergence Theorem \ref{thm:conv:strong} to the fuzzy Boltzmann \eqref{FBE} corresponding to  $\kappa^\sigma$ to obtain the strong convergence (up to a subsequence)
\begin{equation*}
    f^\sigma\to f\quad\text{and}\quad \hf^\sigma\to f\quad\text{in} \quad C([0,T];L^1(\Do)).
\end{equation*}
Moreover, we show that $f$ is indeed a renormalised solution of the inhomogeneous Boltzmann equation  \eqref{IBE}.
\begin{theorem}[\cite{Lio94}, Part II, Theorem 3.2]
\label{thm:conv:strong}
Let $(h^n)\subset C([0,T];L^1(\R^k))$ be a sequence of solutions of
\begin{equation}
\label{transport}
    (\d_t+a(y)\cdot\nabla_y)h^n=H^n-l^nh^n,\quad h^n|_{t=0}=h_0,\quad\text{in}\quad \R^k\times(0,T),
\end{equation}
where $a:\R^k\to\R^k$ satisfies
\begin{equation*}
    |a(y)-a(z)|\le C|y-z|\quad\forall y,\,z\in\R^k.
\end{equation*}
    If
    \begin{enumerate}
        \item For any $t\in[0,T]$, $(h^n(t))$ is relatively weakly compact in $L^1(\R^k)$;
        \item For all $R>0$, we have $H^n\mathbb{1}_{\{|h^n|\le R\}}$ is relatively weakly compact in $L^1((0,T)\times B_R^y)$;
        \item $H^n\ge0$ converges a.e. on $\R^k\times(0,T)$;
        \item For all $R>0$, we have $l^n$ converges in $L^1((0,T)\times B_R^y)$,
    \end{enumerate}
    then $h^n$ converges strongly in $C([0,T];L^1(\R^k))$.
\end{theorem}

We write the fuzzy Boltzmann equations in the form \eqref{transport} by letting
\begin{align*}
&\R^k=\Do,\quad y=(x,v),\quad a(y)=(0,v),\\
&h^n=f^{\sigma_n},\quad H^n=Q^{+}(f^{\sigma_n},\hf^{\sigma_n}),\quad l^n=L(\hf^{\sigma_n}).
\end{align*}
By directly applying Theorem \ref{thm:conv:strong}, we have the following convergence results.
\begin{corollary}
\label{strong:conv}
Let $f^\sigma\in C([0,T]\times\Do)$ be weak solutions of \eqref{FBE} corresponding to $\kappa^\sigma$. Let $f\in C([0,T]\times L^1(\Do))$ be the weak limit (up to a subsequence) of $f^\sigma$ in $L^1([0,T]\times\Do)$. Then we have
\begin{equation*}
    f^\sigma\to f\quad\text{in} \quad C([0,T];L^1(\Do)).
\end{equation*}
\end{corollary}

\begin{proof}

We verify the assumptions in Theorem~\ref{thm:conv:strong}.
\begin{enumerate}
    \item In Lemma~\ref{lemma:f}, we showed that $\{f^\sigma_t\}$ is relatively weakly compact in $L^1(\Do)$.
    \item In Lemma~\ref{lemma:Q}, we showed the relatively weakly compactness of $\frac{Q^{+}(f^\sigma,\hf^\sigma)}{1+f^\sigma}$ in $L^1((0,T)\times\R^d\times B^v_R)$. Notice that $0\le (1+f^\sigma)\mathbb{1}_{\{|f^\sigma|\le R\}
    }\le 1+R$ and $f^\sigma\to f$ in $L^1([0,T]\times\Do)$. Thus, $Q^{+}(f^\sigma,\hf^\sigma)\mathbb{1}_{\{|f^\sigma|\le R\}}$ is relatively weakly compact in $L^1((0,T)\times\R^d\times  B^v_R)$.
    \item In Theorem~\ref{thm:conv:measure}, we showed that $(f^\sigma,\hf^\sigma)\to Q(f,\hf)$ in measure, then there exists a subsequence (still denoted by $f^\sigma$) such that
    \begin{equation*}
    Q^{+}(f^\sigma,\hf^\sigma)\to Q(f,\hf)\quad\text{a.e. on}\quad (0,T)\times B^{x,v}_R.    
    \end{equation*}
    \item In Lemma~\ref{ave:conv:L}, we showed that $ L(\hf^\sigma)\to L(\hf)$
    in $L^1((0,T)\times  B^{x,v}_R)$.
\end{enumerate}
We apply Theorem~\ref{thm:conv:strong} to derive that $f^\sigma$ converges strongly in $C([0,T];L^1(\Do))$. The uniqueness  of weak and strong limits of $f^\sigma$ ensures that
\begin{equation*}
    f^\sigma\to f\quad\text{in} \quad C([0,T];L^1(\Do)).
\end{equation*}
\end{proof}

We recall that $\hf^\sigma\rightharpoonup \hf\in C([0,T];L^1(\Do))$ in $L^1([0,T]\times\Do)$. The strong convergence of $f^\sigma$ ensures that $f=\hf$. Indeed, we have
\begin{equation*}
    \|f-\hf^\sigma\|_{L^1(\Do)}\le \|f-f*_x\kappa^\sigma\|_{L^1(\Do)}+\|f-f^\sigma\|_{L^1(\Do)}\to0
\end{equation*}
as $\sigma\to0$. Then the uniqueness of weak and strong limits of $\hf^\sigma$ implies that $f=\hf$.

We show that $f$ is indeed a renormalised solution of the inhomogeneous Boltzmann equation  \eqref{IBE}.
\begin{theorem}
\label{thm:weak-sol:f-hf}
We have that $f\in C([0,T]\times L^1(\Do))$ is a renormalised solution of 
\begin{equation*}
\label{IBE:h}
    \left\{
    \begin{aligned}
     &(\d_t+v\cdot \nabla_x)f=Q(f,f)\\
     &f|_{t=0}=f_0.
    \end{aligned}
    \right.
\end{equation*}
\end{theorem}

\begin{proof}
We will pass to the limit in the following weak formulation of \eqref{FBE:renorm}
\begin{equation*}
    \begin{aligned}
        &\int_{\Do} \log(1+f_0)\varphi(0)\dd x\dd v+\int_0^T\int_{\Do} \log(1+f^\sigma) \d_t\varphi\dd x\dd v\\
        &+\int_0^T\int_{\Do}v\cdot\nabla_x\varphi \log(1+f^\sigma)+\int_0^T\int_{\Do} \varphi \frac{Q(f^ \sigma,f^\sigma)}{1+f^\sigma}\dd x\dd v=0
    \end{aligned}
\end{equation*}
for all $\varphi\in C^\infty_c([0,T)\times\Do)$.

The convergence $   f^\sigma\to f$ in $C([0,T];L^1(\Do))$ in Corollary~\ref{strong:conv} implies that $\log(1+f^\sigma)\to\log(1+f)$ in $C([0,T];L^1(\Do))$.

The boundedness and pointwise convergence of $\frac{f^\sigma}{1+f^\sigma}\to\frac{f}{1+f}$ combines with the convergence  $L(f^\sigma)\to L(f)$ in $L^1([0,T]\times\R^d\times B^v_R)$, imply 
that 
\begin{equation*}
\frac{Q^-(f^ \sigma,f^\sigma)}{1+f^\sigma}\varphi\to\frac{Q^-(f,f)}{1+f}\varphi \quad\text{in} \quad L^1([0,T]\times \Do).
\end{equation*}

The convergence of the gain term is a direct consequence of the convergence of the loss term
 and the comparison inequality \eqref{relation:f':f}.

\end{proof}

\appendix

\section{Solvability results}
\label{app-sec:sol}
We complete the solvability results in Theorem~\ref{thm:existence} for the fuzzy Boltzmann equation \eqref{pre:FBE} with a fixed $\sigma\in(0,1)$
\begin{equation}
\label{app:FBE}
\left\{
\begin{aligned}
&\d_t f+v\cdot \nabla_x f=Q_{\sf fuz}(f,f)\\
&Q_{\sf fuz}(f,f)=\int_{\Do\times S^{d-1}}
\kappa B(f'f'_*-f_*)\dd x_*\dd v_*\dd \omega.
\end{aligned}
\right.
\end{equation}
For simplification, in the appendix: We drop the parameter $\sigma\in(0,1)$; Since we will only 
Since we only consider the fuzzy collision operator, we take the following abbreviation
\begin{align*}
 f_*=f(x_*,v_*),\quad f'=f(x,v')\quad\text{and}\quad   
 f'_*=f(x_*,v_*').
\end{align*}

Based on \cite{EH25} and the arguments for homogeneous Boltzmann equations, we have the following theorem.

\begin{theorem}
\label{app:thm:existence}
Let the collision kernel $B$ satisfy Assumption~\ref{CK} with $\mu\in(0,1]$. If $f_0\in L^1_{2,2}(\Do;\R_+)$ and  $|\cH(f_0)|<+\infty$, 
    then there exists a global-in-time weak solution $f\in C([0,T];L^1(\Do;\R_+))\cap L^\infty([0,T];L^1_{2,2}(\Do;\R_+))$ of the fuzzy Boltzmann equation \eqref{app:FBE} for any $T>0$ such that the following mass, momentum and energy conservation laws hold
\begin{equation}
\label{app:bdd:L122}
    \begin{aligned}
&\|(1,v,|v|^2)f_t\|_{L^1(\Do)}=\|(1,v,|v|^2)f_0\|_{L^1(\Do)}\quad\forall t\in[0,T],
        \end{aligned}
            \end{equation}
            and the following entropy inequality holds 
\begin{equation}
\label{app:ineq:cH}
    \cH(f_t)-\cH(f_0)\le-\int_0^t \cD(f_s)\dd s\quad\forall t\in[0,T],
\end{equation}
where $\cD(f)$ is defined as in \eqref{def:dissipation}. 

Moreover, if the collision kernel  $B(v-v_*,\omega)=b(\theta)|v-v_*|^\mu$ with $b(\theta)$ bounded and we take the domain $(x,v)\in\T^3\times\R^3$, the corresponding energy-conserved solution of \eqref{app:FBE} is unique.

\end{theorem}
\begin{proof}
In \cite{EH25}, we have proved the above results with the initial value $f_0\in L^1_{2,2+\mu}(\Do)$. The proof can be applied straightforwardly to the initial values $f_0\in L^1_{2,2}(\Do)$, where we have the weak solution $f\in C([0,T];L^1(\Do;\R_+))\cap L^\infty([0,T];L^1_{2,2}(\Do;\R_+))$ and the following energy inequality holds
\begin{equation}
\label{app:en-in}
    \int_{\Do}f_t|v|^2\dd x\dd v\le \int_{\Do}f_0 |v|^2\dd x\dd v\quad\forall t\in[0,T].
\end{equation}

Compared to Theorem~\ref{app:thm:existence}, we still need to show the energy conservation laws and the uniqueness of the energy-conserved solutions on $\T^3\times\R^3$.

\begin{itemize}
    \item {\it Energy conservation laws.} Since the energy of the weak solution is not larger than the initial energy as in \eqref{app:en-in}, to show the energy conservation law, we only need to show 
    \begin{equation}
    \label{app:en-ge}
    \int_{\Do}f_t|v|^2\dd x\dd v\ge \int_{\Do}f_0 |v|^2\dd x\dd v,\quad\forall t\in[0,T].
\end{equation}
We follow the arguments for the homogeneous Boltzmann equation \eqref{HBE} in \cite{Lu99} to show \eqref{app:en-ge}. 

We approximate $|v|^2$ by $\phi_\varepsilon(v)=\varepsilon^{-1}\log(1+\varepsilon|v|^2)$. By using $0\le \log(1+y)\le \sqrt y$, we have $\phi_\varepsilon(v)\le \varepsilon^{-\frac12} \langle v\rangle$. One  can choose $\phi_\varepsilon(v)$ as a test function in the weak formulation and first verify that the following equality is well-defined
\begin{equation}
\label{weak-form:phi-epsi}
\begin{aligned}
  \int_{\Do} f_t \phi_\varepsilon\dd x \dd v=&\int_{\Do}f_0\phi_\varepsilon\dd x \dd v+\int_0^t\int_{\Do}Q_{\sf fuz}(f,f)\phi_\varepsilon \dd x\dd v\dd\tau.    
\end{aligned}
\end{equation}
Since $f\in L^\infty([0,T];L^1_{2,2}(\Do))$, we only need to show the integrability of   $Q^{\pm}_{\sf fuz}(f,f)\phi_\varepsilon$. Concerning the loss term, we have
\begin{align*}
    &\int_{\Do}Q^-_{\sf fuz}(f,f)\phi_\varepsilon\dd x\dd v\\
    =&{}\int_{\G}ff_*\kappa B(v-v_*,\omega)\phi_\varepsilon(v)\dd\sigma\\
\lesssim&{} \varepsilon^{-\frac12}\int_{\R^{4d}}ff_*\langle v_*\rangle^{\mu}\langle v\rangle^{1+\mu}\dd x_*\dd x\dd v_*\dd v\\
\lesssim&{} \varepsilon^{-\frac12}\|f\|_{L^1_{0,2}(\Do)}^2<+\infty.
\end{align*}
The integrability of the gain term $Q^+_{\sf fuz}(f,f)\phi_\varepsilon$ can be obtained by applying the comparison inequality \eqref{relation:f':f}. 

Notice that 
by changing of variables, the collision term can be written as
\begin{align*}
    &\int_{\Do}Q_{\sf fuz}(f,f)\phi_\varepsilon \dd x\dd v\\
    =&{}\frac12\int_{\R^{4d}\times S^{d-1}}\kappa Bff_*\big(\phi_\varepsilon'+(\phi_\varepsilon)_*'-\phi_\varepsilon-(\phi_\varepsilon)_*\big)\dd\sigma.
\end{align*}
By using of the energy conservation law $|v'|^2+|v_*'|^2=|v|^2+|v_*|^2$, the alternating sum of $\phi$ can be written as 
\begin{align*}
&\phi_\varepsilon(v')+\phi_\varepsilon(v_*')-\phi_\varepsilon(v)-\phi_\varepsilon(v_*)\\
=&{}\varepsilon^{-1}\log\frac{(1+\varepsilon|v'|^2)(1+\varepsilon|v_*'|^2)}{(1+\varepsilon|v|^2)(1+\varepsilon|v_*|^2)}\\
=&{}\varepsilon^{-1}\log\Big(1+\frac{\varepsilon^2|v'|^2|v_*'|^2}{1+\varepsilon(|v|^2+|v_*|^2)}\Big)\\
&-\varepsilon^{-1}\log\Big(1+\frac{\varepsilon^2|v|^2|v_*|^2}{1+\varepsilon(|v|^2+|v_*|^2)}\Big).
\end{align*}
Accordingly, we define
\begin{align*}
    &K_\varepsilon(v,v_*)=\frac{1}{2\varepsilon}\int_{S^{d-1}}\log\big(1+\frac{\varepsilon^2|v'|^2|v_*'|^2}{1+\varepsilon(|v|^2+|v_*|^2)}\Big) B(v-v_*,\omega)\dd\omega,\\
    &J_\varepsilon(v,v_*)=\frac{1}{2\varepsilon}\log\Big(1+\frac{\varepsilon^2|v|^2|v_*|^2}{1+\varepsilon(|v|^2+|v_*|^2)}\Big)\int_{S^{d-1}} B(v-v_*,\omega)\dd\omega.
\end{align*}
Then, the weak collision term can be written as
\begin{align*}
&\int_{\Do}Q_{\sf fuz}(f,f)\phi_\varepsilon \dd x\dd v \\
=&{}\int_{\R^{4d}}\kappa ff_*K_\varepsilon\dd x\dd x_*\dd v\dd v_*-\int_{\R^{4d}}\kappa ff_*J_\varepsilon\dd x\dd x_*\dd v\dd v_*.
\end{align*}

We pass to the limit by letting $\varepsilon\to0$ in the weak formulation \eqref{weak-form:phi-epsi} i.e.
\begin{equation*}
\begin{aligned}
  \int_{\Do} f_t \phi_\varepsilon+\int_0^t\int_{\R^{4d}}\kappa f f_* J_\varepsilon=\int_{\Do}f_0\phi_\varepsilon+\int_0^t\int_{\R^{4d}}\kappa f f_* K_\varepsilon.   
\end{aligned}
\end{equation*}

Since $0\le \phi_\varepsilon(v)\le |v|^2$, we use dominate convergence theorem
\begin{equation*}
    \lim_{\varepsilon\to0}\int_{\Do} f_t \phi_\varepsilon\dd x \dd v=\int_{\Do} f_t |v|^2\dd x \dd v.
\end{equation*}
By using of $0\le \log(1+y)\le \sqrt y$, $J_\varepsilon$ is bounded by
\begin{align*}
  0&\le J_\varepsilon(v,v_*)\\
  &\lesssim \varepsilon^{-1}\int_{S^{d-1}}\log\big(1+\frac{\varepsilon^2|v|^2|v_*|^2}{1+\varepsilon(|v|^2+|v_*|^2)}\big)\langle v-v_*\rangle^{\mu}\dd\omega\\
  &\lesssim \frac{|v||v_*|}{\sqrt{1+\varepsilon(|v|^2+|v_*|^2)}}\langle v-v_*\rangle^{\mu}\\
  &\lesssim \langle v\rangle^{1+\mu}\langle v_*\rangle^{1+\mu}.
\end{align*}
 Since $0\le J_\varepsilon\le \langle v\rangle^2\langle v_*\rangle^2$ and $\lim_{\varepsilon\to0}J_\varepsilon=0$, the dominate convergence theorem implies that
\begin{equation*}
\begin{aligned} \int_0^t\int_{\R^{4d}}\kappa f f_* J_\varepsilon\to0\quad\text{as }\varepsilon\to0.
\end{aligned}
\end{equation*}
Hence, we have
\begin{equation*}
\begin{aligned}
  &\int_{\Do} f_t |v|^2\le \int_{\Do}f_0|v|^2+\liminf_{\varepsilon\to0}\int_0^t\int_{\R^{4d}}\kappa f f_*K_\varepsilon.  
\end{aligned}
\end{equation*}
The positivity of $K_\varepsilon$ ensures  \eqref{app:en-ge} that the energy is not smaller than the initial energy. 

Combining with the energy inequality \eqref{app:en-in}, we have the energy conservation law
\begin{equation*}
    \int_{\Do}f_t|v|^2\dd x\dd v= \int_{\Do}f_0|v|^2\dd x\dd v\quad\forall t\in[0,T].
\end{equation*}

    \item {\it Uniqueness of energy-conserved solutions on $\T^3\times\R^3$.} We follow the uniqueness argument for the energy-conserved solutions to the homogeneous Boltzmann equation \eqref{HBE} in \cite{MW99}. On torus $\T^3$, the spatial kernel $\kappa$ has positive upper and lower bounds
    \begin{equation}
    \label{app:kappa:bdd}
        C_\kappa^{-1}\le \kappa(z)\le C_\kappa\quad\forall z\in\T^3
    \end{equation}
   for some constant $C_\kappa>0$.    
    We take the collision kernel the form  $B(v-v_*,\omega)=b(\theta)|v-v_*|^\mu$, $\mu\in(0,1]$.

   The following 
    shape Povzner's inequality has been shown in \cite{MW99} will keep being used in the proof.
     \begin{lemma}[\cite{MW99}]
        \label{app:lem:dep}
            For a function $\Psi=\Psi(|v|^2)\in\R$, we define
            \begin{equation*}
                K(v,v_*)=\int_{S^{d-1}}b(\theta)(\Psi'+\Psi'_*-\Psi-\Psi_*)\dd\omega.
            \end{equation*}
            Then we have
            \begin{equation*}
                K(v,v_*)=G(v,v_*)-H(v,v_*),
            \end{equation*}
            where let $\bar b(\theta)=b(\theta)\sin\theta\cos\theta$ and $H$ is given by
            \begin{align*}
            &H(v,v_*)=8\pi\int_0^{{\pi}/{2}}\big(\bar b(\theta)+\bar b({\pi}/{2}-\theta)\big)\times\\
            &\quad \big(\Psi(|v|^2)\cos^2\theta+\Psi(|v_*|^2)\sin^2\theta-\Psi(|v|^2\cos^2\theta+|v_*|^2\sin^2\theta)\big)\dd\theta.
            \end{align*}
           \begin{enumerate}
               \item Let $\Psi$ be a positive convex function that can be written as $\Psi(x)=x\Phi(x)$, where $\Phi$ is  concave, increasing to infinity, and $\Phi'(x)\lesssim \frac{1}{1+x}$. Then $G$ is bounded by \begin{equation*}
                   |G(v,v_*)|\lesssim |v||v_*|.
               \end{equation*}
               Let $\chi_1(v,v_*)=1-\mathbb{1}_{\{{|v|}/{2}\le |v_*|\le 2|v|\}}$. For any $\varepsilon>0$,
               we have
            \begin{equation*}
                   H(v,v_*)\gtrsim  (|v|^{2-\varepsilon}+|v_*|^{2-\varepsilon})\chi_1(v,v_*).
               \end{equation*}

               \item 
               Let $\Psi(x)=x^{1+r}$ for some $r>0$. 
               Then $G$ is bounded by
               \begin{equation*}
                   |G(v,v_*)|\lesssim r(|v||v_*|)^{1+r},
               \end{equation*}
               and 
            \begin{equation*}
                   H(v,v_*)\gtrsim  r(|v|^{2(1+r)}+|v_*|^{2(1+r)})\chi_1(v,v_*).
               \end{equation*}
              
           \end{enumerate}
              
        \end{lemma}

     To show the uniqueness, in principle, we are estimating the $2+\mu$-moment of the fuzzy Boltzmann equation \eqref{app:FBE}. Let $f$ be an energy-conserved solution of \eqref{app:FBE}. We follow \cite{MW99} to show the uniqueness in four steps: In Step $1$, we improve the integrability that for some convex function $\Psi$, we have the integrability of  $\Psi(|v|^2)f$; In Step $2$ and $3$, we discuss the integrability of $\langle v\rangle^{2+\mu} f$; We show the uniqueness in Step $4$.

    For $s\ge0$, we define 
     \begin{align*}
         M_s(f)=\int_{\T^3\times\R^3} |v|^sf\dd x\dd v.
     \end{align*}
   
    \begin{itemize}
        \item {\bf Step $1$: Integrability of $\Psi(|v|^2)f$.}
For $|v|^2f_0\in L^1(\T^3\times\R^3)$, there exists a convex function
$\Psi(x)$ satisfying the assumptions Lemma~\ref{app:lem:dep}-(1) such that $\Psi(|v|^2)f_0\in L^1(\T^3\times\R^3)$.  The construction of $\Psi$ can be found in \cite[Appendix]{MW99}.

        We will show that for any $t\in[0,T]$, we have the uniform bounds
        \begin{equation}
            \label{app:Psi:bdd}
            C_1\le \int_{\T^3\times\R^3}\Psi(|v|^2)f_t(x,v)\dd x \dd v\le C_2\quad\forall t\in[0,T],
        \end{equation}
        and
         \begin{equation}
         \label{app:L1:2+mu}  \int_0^tM_{2+\mu/2}(f_\tau)\dd\tau\le C(1+t).
        \end{equation}

The bound \eqref{app:Psi:bdd} holds for $t=0$ by definition of $\Psi$. We will show that the bounds \eqref{app:Psi:bdd} hold for all $t\in[0,T]$. We first approximate $\Psi$ by $\Psi_n:\R_+\to\R$
        \begin{equation*}
\Psi_n(x)=\Psi(x)\mathbb{1}_{\{0\le x\le n\}}+p_n(x)\mathbb{1}_{\{n\le x\}},
        \end{equation*}
        where $p_n(x)=(x-n)\Psi'(n)+\Psi(n)$ is linear. Notice that $\Psi_n$ is convex and $\supp(\Psi_n-p_n)\subset[0,n]$.  We write $\Psi_n=\Psi_n(|v|^2)$ and $p_n=p_n(|v|^2)$. We  substitute $\Psi_n-p_n$ to the weak formulation of \eqref{app:FBE} to derive
        \begin{equation}
        \label{app:weak:Psi:p}
        \begin{aligned}
            &\int_{\T^3\times \R^3} (f_t-f_0)\Psi_n
            =\int_{\T^3\times \R^3} (f_t-f_0)(\Psi_n-p_n\big)\\
    =&{}\int_0^t\int_{\T^3\times \R^3} Q_{\sf fuz}(f,f)(\Psi_n-p_n)\\
         =&{}\frac12\int_0^t\int_{(\T^3\times \R^3)^2\times S^{2}} \kappa ff_*\big(\Psi_n'+(\Psi_n')_*-\Psi_n-(\Psi_n)_*\big)|v-v_*|^\mu b(\theta),
        \end{aligned}
        \end{equation}
        where the energy conservation law ensures  $\bar\nabla p_n(|v|^2)=0$. 

      We define $K_{\Psi_n}(v,v_*)$ as in Lemma~\ref{app:lem:dep}, and correspondingly, we have the following decomposition
        \begin{equation*}
            K_{\Psi_n}(v,v_*)= G_{\Psi_n}(v,v_*)-H_{\Psi_n}(v,v_*).
        \end{equation*}
       We substitute $K_{\Psi_n}$ to  \eqref{app:weak:Psi:p} 
        to derive
            \begin{equation}
            \label{app:S1:GH-n}
        \begin{aligned}
            &\int_{\T^3\times \R^3} f_t\Psi_n+\frac12\int_0^t\int_{(\T^3\times \R^3)^2} \kappa ff_*|v-v_*|^\mu H_{\Psi_n}\\
=&{}\int_{\T^3\times \R^3} f_0\Psi_n+\frac12\int_0^t\int_{(\T^3\times \R^3)^2} \kappa ff_*|v-v_*|^\mu G_{\Psi_n}.
        \end{aligned}
        \end{equation}

        To show the integrability of $\Psi(|v|^2)f$, we pass to the limit by letting $n\to\infty$ in \eqref{app:S1:GH-n}. 
Notice that $\Psi_n$ and $H_{\Psi_n}$ monotonically converge to $\Psi$ and $H_{\Psi_n}$. The Lemma \ref{app:lem:dep}-(1) ensures that $|G_{\Psi_n}|\lesssim |v||v_*|$ for all $n\in\N_+$. Let $n\to\infty$ in \eqref{app:S1:GH-n}, we have
\begin{equation}
\label{app:S1:GH-n:2}
        \begin{aligned}
&\int_{\T^3\times \R^3} f_t\Psi+\frac12\int_0^t\int_{(\T^3\times \R^3)^2} \kappa ff_*|v-v_*|^\mu H_{\Psi}\\
     \le&{}\int_{\T^3\times \R^3} f_0\Psi+C\int_0^t\int_{(\T^3\times \R^3)^2} ff_*|v-v_*|^\mu |v||v_*|.
        \end{aligned}
        \end{equation}
Lemma \eqref{app:lem:dep}-(1) implies that $ H(v,v_*)\gtrsim  (|v|^{2-\frac{\mu}{2}}+|v_*|^{2-\frac{\mu}{2}})\chi_1(v,v_*)$, and the lower bound of $\kappa$ is given by \eqref{app:kappa:bdd}. Hence, we have
\begin{align*}
   &\kappa ff_* |v-v_*|^\mu H_\Psi\\
   \ge&{} ff_*\big(C'|v|^{2+\frac{\mu}{2}}-C''(|v||v_*|)^{1+\frac{\mu}{4}}\mathbb{1}_{\{{|v|}/{2}\le |v_*|\le 2|v|\}}\big)
\end{align*}
for some constants $C',C''>0$.

We substitute the lower bounds on $H_{\Psi}$ to \eqref{app:S1:GH-n:2} to derive
\begin{equation*}
        \label{app:ineq:Psi}
        \begin{aligned}
            &\int_{\T^3\times \R^3} f_t\Psi+C'\int_0^t\int_{\T^3\times \R^3} f|v|^{2+\frac{\mu}{2}} \\
\le&{}\int_{\T^3\times \R^3} f_0\Psi+C\int_0^t\int_{(\T^3\times \R^3)^2} ff_* \big(\langle v\rangle\langle v_*\rangle\big)^{1+\mu}\\
\lesssim &{} t+1,
        \end{aligned}
        \end{equation*}
where we use $0< \mu\le 1$ and $\||v|^{1+\mu}f\|_{L^1(\T^3\times\R^3)}\le \|\langle v\rangle^2f\|_{L^1(\T^3\times\R^3)}$.

We conclude the bounds \eqref{app:Psi:bdd} and \eqref{app:L1:2+mu}.

        \item {\bf Step $2$: Integrability of $ |v|^{s}f$, $s>2$} We will show for any $t_0>0$ and $s>2$, we have 
        \begin{equation}
            \label{app:s:bdd}
            \sup_{t\ge t_0}M_s(f_t)\le C(t_0,s).
        \end{equation}
       The bound \eqref{app:L1:2+mu} i.e. $\int_0^tM_{2+{\mu}/{2}}(f_\tau)\le C$ implies that 
        there exists at least a sequence $\{t_n\}$ such that $t_n\to0$ and 
         \begin{equation*}
M_{2+{\mu}/{2}}(f_{t_n})\le C\quad\forall n\in\N_+.
        \end{equation*}
        For any fixed $n\in\N_+$, starting from $f_{t_n}$, we show the uniform bounds of $M_{2+{\mu}/{2}}(f_{t})$ for all $t\in[t_n,T]$ by repeating the argument in Step 1. We take $\Psi(x)=x^{1+\frac{\mu}{4}}$ (take approximation $\Psi_n$ as in Step $1$, and pass to the limit) and apply Lemma \ref{app:lem:dep}-$(2)$ to derive
               \begin{equation}
        \label{app:weak:1+mu4}
        \begin{aligned}
            &\int_{\T^3\times \R^3} f_t|v|^{2+\frac{\mu}{2}}+\frac12\int_{t_n}^t\int_{(\T^3\times \R^3)^2} \kappa ff_*|v-v_*|^\mu H_{\Psi}\\
\le&{}\int_{\T^3\times \R^3} f_{t_n}|v|^{2+\frac{\mu}{2}}+C\int_{t_n}^t\int_{(\T^3\times \R^3)^2} ff_*|v-v_*|^\mu \big(|v||v_*|\big)^{1+\frac{\mu}{4}}\\
            \le&{}\int_{\T^3\times \R^3} f_{t_n}|v|^{2+\frac{\mu}{2}}+C\|f\|_{L^1_{0,2}(\T^3\times\R^3)}\int_{0}^t\|f\|_{L^1_{0,2+\frac{\mu}{2}}(\T^3\times\R^3)}\\
            \le&{} M_{2+\mu/2}(f_{t_n})+C(t+1),
        \end{aligned}
        \end{equation}
        where $C>0$ depending only on $\|f\|_{L^1_{0,2}(\Do)}$.
        
Lemma \ref{app:lem:dep}-$(2)$ implies that 
\begin{align*}
   & \kappa ff_*|v-v_*|^\mu H_\Psi\\
   \ge&{} ff_*\big(C'|v|^{2+\frac{3\mu}{2}}-C''(|v||v_*|)^{1+\frac{3\mu}{4}}\mathbb{1}_{\{\frac{|v|}{2}\le |v_*|\le 2|v|\}}\big).
\end{align*} 
We substitute the above lower bound to \eqref{app:weak:1+mu4} to derive
               \begin{equation*}
        \begin{aligned}
            &\int_{\T^3\times \R^3} f_t|v|^{2+\frac{\mu}{2}}+C'\int_{t_n}^t\int_{\T^3\times \R^3} f|v|^{2+\frac{3\mu}{2}}
            \le C(t+1)\quad\forall t\in[t_n,T].
        \end{aligned}
        \end{equation*} 
    By induction, we have, for all $s>2$
              \begin{equation}
                \label{app:bdd:2:1+t}  
                \begin{aligned}
     M_s(f_t)+C'\int_{t_n}^t  M_{s+\mu}(f_\tau)\dd\tau\le C(t+1)\quad t\in[t_n,T].
                \end{aligned}
              \end{equation}

           By choosing $t_n$ arbitrarily small, we conclude with the uniform bound
            \eqref{app:s:bdd}.

        \item {\bf Step $3$: $M_{2+\mu}(f_t)$ near $t=0$.}
        We will show that 
        there exists a function $\gamma(t)\to0$ as $t\to0$ such that
        \begin{equation}
         \label{app:Y:gamma/t}   M_{2+\mu}(f_t)\le {\gamma(t)}/{t}.
        \end{equation}
 We recall \eqref{app:bdd:2:1+t} and choose $s={2+\mu}$. We write $M_s=M_s(f_t)$. Then we have the differential inequality
\begin{equation}
        \label{app:diff:Y:2mu}
            \frac{d}{dt}M_{2+\mu}\le C-C'M_{2+2\mu} \quad \forall t\in[t_0,T],
        \end{equation}
        where $C>0$ depending only on $\|f\|_{L^1_{0,2}(\Do)}$ similar to \eqref{app:weak:1+mu4}.

        To bound of $M_{2+\mu}$, we bound  $M_{2+2\mu}$ from below by $M_{2+\mu}$. Notice that
\begin{equation*}
    M_{2+\mu}(f)=\int_{\T^3\times\R^3}f(x,v)\Psi(|v|^2)\frac{|v|^{2+\mu}}{\Psi(|v|^2)}\dd x \dd v.
\end{equation*}
We recall the convex function $\Psi$ in Step $1$. 
        We define the function 
        \begin{equation*}
            \beta(x)=x\Psi(x).
        \end{equation*}
    The convexity and monotonicity of $\Psi$ ensure that $|v|^{2+\mu}\le \Psi(|v|^2)\Psi^{-1}(|v|^\mu)$, and hence, 
        \begin{equation}
    \label{app:lbdd:Y}\Psi(|v|^2)\beta\Big(\frac{|v|^{2+\mu}}{\Psi(|v|^2)}\Big)=|v|^{2+\mu}\Psi\Big(\frac{|v|^{2+\mu}}{\Psi(|v|^2)}\Big)\le |v|^{2+2\mu}.
\end{equation}
We recall the bound \eqref{app:Psi:bdd} in step 1 that $M_\Psi:=\int_{\T^3\times\R^3}f \Psi(|v|^2)\dd x \dd v$ and $C_1\le M_\Psi\le C_2$.
Notice that $\beta$ is a convex function. By Jensen's inequality, we have  
\begin{align*}
&\beta\big(M_\Psi^{-1}M_{2+\mu}(f)\big)\\
=&{}\beta\big(M_\Psi^{-1}\int_{\T^3\times\R^3}\frac{|v|^{2+\mu}}{\Psi(|v|^2)}f\Psi(|v|^2)\dd x \dd v\big)\\
\le&{} M_\Psi^{-1}\int_{\T^3\times\R^3}\beta\Big(\frac{|v|^{2+\mu}}{\Psi(|v|^2)}\Big)f\Psi(|v|^2)\dd x \dd v\\
\le&{} M_\Psi^{-1} M_{2+\mu}(f),
\end{align*}
where we use the bound \eqref{app:lbdd:Y} for the last inequality.
Since $\beta$ is non-decreasing and $C_1\le  M_\Psi\le C_2$, we have 
\begin{equation*}
C_1\beta(C_2^{-1}M_{2+\mu})\le M_{2+2\mu}.
\end{equation*}

We substitute the above lower bound of $M_{2+2\mu}$ to \eqref{app:diff:Y:2mu} to derive
\begin{equation*}
    \frac{d}{dt} M_{2+\mu}\le C-C'\beta(C_2^{-1}M_{2+\mu})\quad\forall t\in[t_0,T].
\end{equation*}

Since $\frac{d}{dt}M_{2+\mu}\le C$ and $\langle v\rangle^{2+\mu}f_0\notin L^1(\T^3\times\R^3)$, we choose $t_0$ and $\bar t$ small enough such that $C\le \frac{C'}{2}\beta(C_2^{-1}M_{2+\mu})$ for all $t\in [t_0,\bar t]$, and 
\begin{equation}
\label{app:diff-2}
    \frac{d}{dt} M_{2+\mu}\le -\frac{C'}{2}\beta(C_2^{-1}M_{2+\mu})\quad \forall t\in(0,\bar t].
\end{equation}

We define 
\begin{equation*} \Gamma(y)=\int_{y}^{+\infty}\beta(C_2^{-1}{z})^{-1}\dd z.
\end{equation*}
Then \eqref{app:diff-2} implies that
\begin{align*}
    &\frac{d}{dt}\Gamma(M_{2+\mu})=-\beta( C_2^{-1}{M_{2+\mu}})^{-1}\frac{d}{dt}M_{2+\mu}\ge \frac{C'}{2}\quad\forall t\in(0,\bar t].
\end{align*}
Since $M_{2+\mu}(f_0)=+\infty$ and $\Gamma\big(M_{2+\mu}(f_0)\big)=0$, we have 
\begin{equation*}
    \Gamma(Y_{2+\mu})\ge \frac{C'}{2} t\quad\forall t\in(0,\bar t].
\end{equation*}
By definition $\Gamma(y)\Psi(C_2^{-1}y)\le 1$, then we have $\Psi(C_2^{-1}M_{2+\mu})\le \frac{2}{C' t}$. The monotonicity of $\Psi$ ensures that $M_{2+\mu}\le C_2\Psi^{-1}\big(\frac{2}{C' t}\big)$. We conclude with \eqref{app:Y:gamma/t} by choosing $\gamma(t)\defeq C_2t\Psi^{-1}\big(\frac{2}{C't}\big)$.

        \item {\bf Step $4$: Uniqueness.} Let $f$ and $g$ both be energy-conserved solutions to fuzzy Boltzmann equation \eqref{pre:FBE} with the same initial value $f_0$.  The difference equation can be written as
        \begin{equation*}
            \label{app:diff}\left\{
            \begin{aligned}
            &(\d_t+v\cdot\nabla_x) (f-g)=Q_{\sf fuz}(f,f)-Q_{\sf fuz}(g,g) \\
            &f-g|_{t=0}=0,
            \end{aligned}
             \right.
        \end{equation*}
        where 
        \begin{align*}
        &Q_{\sf fuz}(f,f)-Q_{\sf fuz}(g,g)   \\
        &= \frac12\int_{\T^3\times \R^3\times S^2} \kappa B\big((f'-g')(f'_*+g'_*)+(f'+g')(f'_*-g_*')\\
        &\quad-(f-g)(f_*+g_*)-(f+g)(f_*-g_*)\big)\dd x_*\dd v_*\dd\omega.
        \end{align*}
        We test the above equation by $\sgn(f-g)$ and $\langle v\rangle^2\sgn(f-g)$ to derive
        \begin{align*}
            \frac{d}{dt}\int_{\T^3\times\R^3}|f-g|&\le \int_{(\T^3\times\R^3)^2\times S^{2}}\kappa B(f+g)|f_*-g_*|\\
            \frac{d}{dt}\int_{\T^3\times\R^3}\langle v\rangle^2|f-g|&\le \int_{(\T^3\times\R^3)^2\times S^{2}}\kappa B\langle v\rangle^2(f+g)|f_*-g_*|.
        \end{align*}
        We define $X_1=\int_{\T^3\times\R^3}|f-g|$ and $X_2=\int_{\T^3\times\R^3}|f-g|\langle v\rangle^2$. Then the differential inequalities hold
        \begin{equation}
        \label{app:X1:X2}
        \begin{aligned}
            &\frac{d}{dt}X_1(t)\lesssim \|f\|_{L^1_{0,2}(\T^3\times\R^3)}X_2(t),\\
            &\frac{d}{dt}X_2(t)\lesssim X_1(t)\int_{\T^3\times\R^3}
            (f+g)\langle v\rangle^{2+\mu}+\|f\|_{L^1_{0,2}(\T^3\times\R^3)}X_2(t).
        \end{aligned}
            \end{equation}
         Notice that $\|f\|_{L^1_{0,2}(\T^3\times\R^3)}=\|g\|_{L^1_{0,2}(\T^3\times\R^3)}$ and
        \begin{equation*}
            0\le X_1(t) \le X_2(t)
\le 2\|f\|_{L^1_{0,2}(\T^3\times\R^3)}.
\end{equation*}
Since $X_1(0)=X_2(0)=0$, we derived the linear growth bound on $X_1(t)\le Ct$. By using the bound \eqref{app:Y:gamma/t} on small enough time interval $[0,\bar t]$ derived in Step $3$, we have 
$$\int_{\T^3\times\R^3}(f+g) |v|^{2+\mu}\dd x\dd v\le \frac{2\gamma(t)}{t},$$
where $\gamma(t)\to0$ as $t\to0$. We then have the linear growth bound $X_2(t)\le Ct$.

We substitute the linear bounds to \eqref{app:X1:X2} to derive the quadratic bounds $X_1(t),X_2(t)\le Ct^2$. Thus, $X_2(t)$ is continuous and differentiable at $t=0$, and $X_2'(0)=0$. By integrating the differential inequality of $X_2(t)$ from $0$ to $t$, one has
\begin{equation}
\label{app:X2:Z}
    X_2(t)\le C\int_0^t\frac{X_2(\tau)}{\tau}\dd\tau.
\end{equation}

We define $Z(t)\defeq C\int_0^t\frac{X_2(\tau)}{\tau}\dd\tau$. Notice that 
$(Z(t)t^{-1})'=(X_2(t)-Z(t))t^{-2}\le 0$. Combining  with the fact $Z(0)=0$, we have $Z(t)\le 0$. Hence,  \eqref{app:X2:Z} implies that $X_2(t)\le0$ for all $t\in[0,\bar t]$. On the other hand, the positivity of $X_2$ implies that $X_2(t)=0$ for all $t\in[0,\bar t]$. This procedure is known as Nagumo's uniqueness criterion, see for example \cite{Nag1926}. On the time interval $[\bar t,T]$, one can use the moment bounds on $M_{2+\mu}$ derived in Step 2 and the differential inequality \eqref{app:X1:X2} to show $X_2(t)=0$. Hence, we conclude the uniqueness that $f_t=g_t$ for all $t\in[0,T]$.

    \end{itemize}
\end{itemize}

\end{proof}

\printbibliography
\end{document}